\documentclass[10pt,a4paper,reqno,twoside]{amsart}

\usepackage[usenames,dvipsnames,svgnames,table]{xcolor}
\usepackage{amssymb}
\usepackage{mathtools}
\DeclarePairedDelimiter{\ceil}{\lceil}{\rceil}
\DeclarePairedDelimiter{\floor}{\lfloor}{\rfloor}

\newcommand{\R}{{\mathbb{R}}}
\newcommand{\N}{{\mathbb{N}}}
\newcommand{\<}{\langle}
\renewcommand{\>}{\rangle}

\newcommand{\abs}[1]{\left\vert#1\right\vert}
\newcommand{\xii}{{\abs{\xi}}}

\newcommand{\supp}{{\mathrm{\,supp\,}}}
\def\<#1\>{\left\langle#1\right\rangle }
\renewcommand{\doteq}{{\,\mathrm{:=}\,}}

\newcommand{\eps}{\varepsilon}

\newcommand{\lin}{{\mathrm{lin}}}
\renewcommand{\hom}{{\mathrm{hom}}}

\newcommand{\sca}{{\mathrm{sc}}}

\theoremstyle{plain}
\newtheorem{theorem}{Theorem}
\newtheorem*{theorem*}{Theorem}
\newtheorem{proposition}{Proposition}[section]
\newtheorem{lemma}{Lemma}[section]
\newtheorem{corollary}{Corollary}[section]
\theoremstyle{definition}

\newtheorem{Not}{Notation}
\theoremstyle{remark}
\newtheorem{remark}{Remark}[section]

\title[The semilinear fractional diffusive equation]{The critical exponent(s) for the semilinear fractional diffusive equation}

\author{Marcello D'Abbicco, Marcelo Rempel Ebert, Tiago Henrique Picon}

\begin{document}

\baselineskip14pt

\begin{abstract}
In this paper we show that there exist two different critical exponents for global small data solutions to the semilinear fractional diffusive equation
\[ \begin{cases}
\partial_t^{1+\alpha} u -\Delta u = |u|^p,& t\geq 0,\ x\in\R^n,\\
u(0,x)=u_0(x), & x\in\R^n,\\
u_t(0,x)=u_1(x) & x\in\R^n,
\end{cases} \]
where~$\alpha\in(0,1)$, and~$\partial_t^{1+\alpha}u$ is the Caputo fractional derivative in time. The second critical exponent appears if the second data is assumed to be zero. This peculiarity is related to the fact that the order of the equation is fractional, and so the role played by the second data~$u_1$ becomes ``unnatural'' as~$\alpha$ decreases to zero. To prove our result, we first derive $L^r-L^q$ linear estimates, $1\leq r\leq q\leq\infty$, for the solution to the linear Cauchy problem, where~$|u|^p$ is replaced by~$f(t,x)$, and then we apply a contraction argument.
\end{abstract}

\keywords{Semilinear partial differential equations, Caputo derivative, Mittag-Leffler functions, fractional derivatives, critical exponents, small data global solutions}

\subjclass[2010]{Primary 35R11; Secondary 35A01, 35B33}

\maketitle

\section{Introduction}\label{sec:intro}

We consider the ``Cauchy type'' problem for the semilinear fractional diffusive equation
\begin{equation}\label{eq:CP}
\begin{cases}
\partial_t^{1+\alpha} u -\Delta u = |u|^p,\qquad t>0,\ x\in\R^n,\\
u(0,x)=u_0(x),\\
u_t(0,x)=u_1(x),
\end{cases}
\end{equation}
where~$\alpha\in(0,1)$ and~$p>1$. Here~$\partial_t^{1+\alpha}u$ is the Caputo fractional derivative of order $1+\alpha$ with respect to~$t$ of~$u(t,x)$, defined by
\begin{equation}\label{eq:caputo}
\partial_t^{j+\alpha} u (t,x) \doteq J^{1-\alpha} (\partial_t^{j+1} u) (t,x)
\end{equation}
for any~$j\in\N$ and~$\alpha\in(0,1)$, where
\begin{equation}\label{eq:RLint}
J^{\beta}f(t) \doteq \frac1{\Gamma(\beta)}\,\int_0^t (t-s)^{\beta-1}\,f(s)\,ds,\qquad t>0,
\end{equation}
is the Riemann-Liouville fractional integral operator~\cite{R}, defined for~$\Re(\beta)>0$. Here~$\Gamma$ is the Euler Gamma function. %Here and in the following, we assume that~$J^\beta f(0)=0$ (for any~$f$ smooth at~$t=0$).
In this paper, we write Cauchy problem instead of ``Cauchy type'' problem, for the sake of brevity.

The semilinear fractional diffusive equation given in~\eqref{eq:CP} interpolates the semilinear heat equation formally obtained at $\alpha=0$ and the semilinear wave equation obtained at~$\alpha=1$. However, as one may expect, the role played by the second data~$u_1$ quickly becomes ``unnatural'' as~$\alpha$ decreases to zero.

The fundamental role played by the second data in influencing the critical exponent for global small data solutions to~\eqref{eq:CP} is a very peculiar effect, which is related to the fact that the order of the equation is fractional. One of the main motivation of our paper is to show and discuss this peculiarity.

By global small data solutions, we mean that for sufficiently small data with respect to some norm, the (unique) solution to~\eqref{eq:CP} is global in time. By critical exponent we mean the power~$\bar p$ such that small data global solutions exist in the supercritical range~$p>\bar p$ (possibly with a bound from above on~$p$), and no global solution exists in the subcritical range~$p\in(1,\bar p)$, under suitable sign assumption on the data.

Assuming small data in~$L^1\cap L^p$, we will prove global existence of the solution to~\eqref{eq:CP}, if~$p\geq\bar p$, where
\begin{equation}\label{eq:pcrit}
\bar p \doteq 1+\frac{2}{n-2(1+\alpha)^{-1}}.
\end{equation}
On the other hand, if the second data~$u_1$ is assumed to be zero and~$u_0$ is small in~$L^1\cap L^p$, then global existence of the solution to~\eqref{eq:CP} holds if~$p>\tilde p$, where
\begin{equation}\label{eq:pcrit0}
\tilde p \doteq 1+\frac{2}{n-2+2(1+\alpha)^{-1}}.
\end{equation}
These two critical exponents are justified by scaling arguments and the nonexistence counterpart result is proved in~\cite{DA+} (see Section~\ref{sec:scaling}).

Even if our interest in~\eqref{eq:CP} is mainly motivated by the mathematical effect on the critical exponent appearing for an equation with fractional order, fractional integrals and derivatives are not just a pure mathematical tool, chosen to study new effects which do not appear in equations with integer order. Fractional integrals and derivatives appear in several models in different areas of science as Biology, Engineering, Mathematical Physics, Medicine with current and unsaturated field survey. The probably most charming aspect of fractional differentiation for the real-world applications is that systems containing fractional derivatives ``keep memory of the past'', and this is a highly desirable property.

We refer to~\cite{KST} for an introduction on the theory of fractional derivatives and to \cite{AMS, FDLS, KL, R2, MT, MMOC} to illustrate some applications.

If $u_{1}\equiv0$ then the solution to~\eqref{eq:CP} may be found solving an integrodifferential equation that is a particular type of Volterra integral equations~\cite{B} (see Section~\ref{sec:caputo}). This problem, that represents the heat conductor model with memory~\cite{GP, M}, was originally studied by Y. Fujita~\cite{F} for $n=1$ (see also~\cite{SW}). Fujita's method produces an handle representation of solution via multiplier operators and it have been used to study~\eqref{eq:CP} in several directions, see \cite{AF, AV, DAEPproc, HC}.

\subsection{A brief story of critical exponents for heat and wave equations}

In his pioneering paper~\cite{Fujita}, H. Fujita consider the Cauchy problem for the semilinear heat equation
\begin{equation}\label{eq:Fujita}
\begin{cases}
\partial_t u -\Delta u = |u|^p,\qquad t\geq0,\ x\in\R^n,\\
u(0,x)=u_0(x),
\end{cases}
\end{equation}
and proved that the power exponent $\tilde p=1+2/n$ is \emph{critical.} In particular, he derived global existence of small data classical solutions in the supercritical range~$p>\tilde p$, and a finite time blow up behavior of solutions in the subcritical range~$p\in(1,\tilde p)$. A blow up result for $p=\tilde p$ has been proved in \cite{Hayakawa, Kobayashietal}. In presence of nonlinear memory terms, that is, when the power nonlinearity~$|u|^p$ is replaced by~$J^\beta(|u|^p)$ in~\eqref{eq:Fujita} (here~$J^\beta$ is as in~\eqref{eq:RLint}), the critical exponent has been obtained in~\cite{CDW} (see later, Section~\ref{sec:scaling}).

On the other hand, the nonexistence of global generalized solutions to the Cauchy problem for the semilinear wave equation
\begin{equation}\label{eq:CPwave}
\begin{cases}
\partial_t^{2} u -\Delta u = |u|^p,\qquad t\geq0,\ x\in\R^n,\\
u(0,x)=u_0(x),\\
u_t(0,x)=u_1(x),
\end{cases}
\end{equation}
has been proved for $1<p<\bar p$, where~$\bar p=1+2/(n-1)$ (and for any~$p>1$ if~$n=1$) by Kato~\cite{K}. However, the critical exponent for~\eqref{eq:CPwave} is known to be~$p_0(n)$, the positive root of the quadratic equation
\begin{equation}\label{eq:strauss}
(n-1)p^2-(n+1)p-2=0,
\end{equation}
as conjectured by Strauss~\cite{Strauss}, after that John proved it in space dimension~$n=3$~\cite{J}. Several authors studied the problem in different space dimension, finding blow-up in finite time for a suitable choice of initial data in the subcritical range~\cite{G, JZ, Sc, Si, YZ06}), and global existence of small data solutions in the supercritical range~\cite{EKP, GLS, G2, T, Z}.

The critical exponent of the Cauchy problem for the semilinear wave equation becomes Fujita exponent~$1+2/n$ if a damping term~$u_t$ is added to the equation in~\eqref{eq:CPwave} (see~\cite{IMN04, N04, TY01, Z01}). This effect is a consequence of the diffusion phenomenon: the asymptotic profile as~$t\to\infty$ of the solution to the damped wave is described by the solution to a heat equation. The situation remains the same if the damping term~$b(t)u_t$ is added to the equation in~\eqref{eq:CPwave}, for a quite large class of coefficients~$b(t)$ (see~\cite{DA15, DALR13, LNZ12, W14}). However, an interesting transition model has been found and studied in~\cite{DAL15, DALR15}: if~$b(t)=2/(1+t)$, the critical exponent is given by~$1+2/n$ if~$n=1,2$ and by~$p_0(n+2)$ if~$n\geq3$, is odd (here~$p_0$ is as in~\eqref{eq:strauss}).

\subsection{Results}

Having in mind our plan to apply a contraction argument to prove the global existence of small data solutions to~\eqref{eq:CP}, we first derive sharp $L^r-L^q$ estimates, with~$1\leq r\leq q\leq\infty$, for the solution to the linear problem:
\begin{equation}\label{eq:CPlinear}
\begin{cases}
\partial_t^{1+\alpha} u -\Delta u = f(t,x),\qquad t>0,\ x\in\R^n,\\
u(0,x)=u_0(x),\\
u_t(0,x)=u_1(x).
\end{cases}
\end{equation}
Therefore, our first result is the following.
\begin{theorem}\label{thm:linear}
Let~$n\geq1$ and~$q\in[1,\infty]$. Assume that~$u_0\in L^{r_0}$, $u_1\in L^{r_1}$, and that~$f(t,\cdot)\in L^{r_2}$, with~$r_j\in[1,q]$, satisfying
\begin{equation}\label{eq:r}
\frac{n}2\left(\frac1{r_j}-\frac1q\right) < 1,
\end{equation}
for~$j=0,1,2$. Assume that
\begin{equation}\label{eq:fgood}
\|f(t,\cdot)\|_{L^{r_2}}\leq K\,(1+t)^{-\eta}, \qquad \forall t\geq0,
\end{equation}
for some~$K>0$ and~$\eta\in\R$. Then the solution to~\eqref{eq:CPlinear} verifies the following estimate:
\begin{align*}
\|u(t,\cdot)\|_{L^q}
    & \leq C\, t^{-\frac{n(1+\alpha)}2\left(\frac1{r_0}-\frac1q\right)}\, \|u_0\|_{L^{r_0}}+ C\,t^{1-\frac{n(1+\alpha)}2\left(\frac1{r_1}-\frac1q\right)}\, \|u_1\|_{L^{r_1}}\\
    & \qquad + \begin{cases}
    CK\,(1+t)^{\alpha-\frac{n(1+\alpha)}2\left(\frac1{r_2}-\frac1q\right)} & \text{if~$\eta>1$,} \\
    CK\,(1+t)^{\alpha-\frac{n(1+\alpha)}2\left(\frac1{r_2}-\frac1q\right)}\,\log(1+t) & \text{if~$\eta=1$,} \\
    CK\,(1+t)^{1-\eta+\alpha-\frac{n(1+\alpha)}2\left(\frac1{r_2}-\frac1q\right)} & \text{if~$\eta<1$,}
    \end{cases}
\end{align*}
for any~$t>0$, where~$C$ does not depend on the data.
\end{theorem}
\begin{remark}
In particular, problem~\eqref{eq:CPlinear} is well-posed in~$L^q$ if~$f(t,\cdot)\in L^q$, and~\eqref{eq:fgood} holds for some~$\eta$ and for~$r_2=q$, since
\begin{equation}\label{eq:wp}
\|u(t,\cdot)\|_{L^q}\leq C\bigl(\|u_0\|_{L^q}+t\|u_1\|_{L^q}\bigr)+C_T\,K,
\end{equation}
for any~$t\in[0,T]$, for some~$C_T>0$.
\end{remark}
Using~\eqref{eq:wp} for~$T=1$, and applying Theorem~\ref{thm:linear} with~$r_j=1$ if~$q<1+2/(n-2)$, or~$r_j$ defined by
\[ \frac{n}2\left(\frac1{r_j}-\frac1q\right)= 1 - \frac\delta{1+\alpha}, \]
otherwise, for~$j=0,1,2$, for any~$t\geq1$, and for some~$\delta>0$, one derives the following, immediate corollary of Theorem~\ref{thm:linear}, by standard contraction and prolongation arguments for linear problems.
\begin{corollary}\label{cor:linear}
Assume that~$u_0,u_1\in L^1\cap L^p$, for some~$p\in[1,\infty]$, and that~\eqref{eq:fgood} holds for any~$r_2\in[1,p]$, for some~$\eta>1$. Then there exists a unique solution~$u\in\mathcal C([0,\infty),L^1\cap L^p)$ to \eqref{eq:CPlinear}, and, for any~$\delta>0$, it satisfies the following estimate
\begin{equation}
\label{eq:decayulin}
\|u(t,\cdot)\|_{L^q} \leq C\,(1+t)^{-\beta_q}\,\bigl(\|u_0\|_{L^1}+\|u_0\|_{L^q}+t\,\|u_1\|_{L^1}+t\|u_1\|_{L^q}+(1+t)^\alpha\,K\bigr), \quad \forall \, q\in[1,p], \ \forall t\geq0,
\end{equation}
where
\begin{equation}\label{eq:betaq}
\beta_q \doteq \min \left\{ \frac{n}2(1+\alpha)\left(1-\frac1q\right), 1+\alpha-\delta \right\}.
\end{equation}
\end{corollary}
\begin{remark}
Taking a sufficiently small~$\delta>0$, one may choose
\begin{equation}\label{eq:betaqgood}
\beta_q=\frac{n}2(1+\alpha)\left(1-\frac1q\right)
\end{equation}
in~\eqref{eq:betaq}, provided that~$q\neq\infty$ if~$n=2$ and~$q<1+2/(n-2)$ if~$n\geq3$.
\end{remark}
\begin{remark}\label{rem:decayhom}
The decay rate for~$\|u(t,\cdot)\|_{L^q}$ in~\eqref{eq:decayulin} is given by~$(1+t)^{1-\beta_q}$, provided that~$u_1$ is nontrivial. If~$u_1$ identically vanishes, and~$f$ is nontrivial in~\eqref{eq:fgood} with~$\eta>1$, then the decay rate is given by~$(1+t)^{\alpha-\beta_q}$. This latter is worse than the decay rate for~\eqref{eq:CPlinear}, in the case in which both~$u_1$ and~$f$ identically vanish. This phenomenon is related once again to the fractional order of integration.
\end{remark}
Theorem~\ref{thm:linear} is the key tool to prove the following small data global existence results. For the sake of simplicity, we will assume~$p<1+2/(n-2)$, so that one may assume~$\beta_q$ as in~\eqref{eq:betaqgood} for any~$q\in[1,p]$ (see later, Section~\ref{sec:Linf}).
\begin{theorem}\label{thm:10}
Let~$n\geq2$ and $p\geq\bar p$, in~\eqref{eq:CP}, with~$\bar p$ as in~\eqref{eq:pcrit}. Moreover, let~$p<1+2/(n-2)$ if~$n\geq3$. Then there exists~$\eps>0$ such that for any~$u_0,u_1\in L^1\cap L^p$, satisfying
\begin{align}
\label{eq:u0eps}
\|u_0\|_{L^1\cap L^p}
    & \doteq \|u_0\|_{L^1}+\|u_0\|_{L^p}\leq \eps,\\
\label{eq:u1eps}
\|u_1\|_{L^1\cap L^p}
    & \doteq \|u_1\|_{L^1}+\|u_1\|_{L^p}\leq \eps,
\end{align}
there exists a unique global solution
\begin{equation}\label{eq:solution0}
u\in\mathcal C([0,\infty),L^1\cap L^p)
\end{equation}
to~\eqref{eq:CP}. Moreover, the solution satisfies the decay estimate
\begin{equation}
\label{eq:decayu1}
\|u(t,\cdot)\|_{L^q} \leq C\,(1+t)^{1-\frac{n}2(1+\alpha)\left(1-\frac1q\right)}\,\bigl(\|u_0\|_{L^1\cap L^p}+\|u_1\|_{L^1\cap L^p}\bigr), \qquad \forall \, q\in[1,p], \ \forall t\geq0,
\end{equation}
where~$C>0$ does not depend on the data.
\end{theorem}
\begin{remark}\label{rem:decay1}
The decay rate in~\eqref{eq:decayu1} is the same as the decay rate of the linear problem, in~\eqref{eq:decayulin}, provided that~$u_1$ is non-trivial (see Remark~\ref{rem:decayhom}).
\end{remark}
If the second data~$u_1$ is zero, then the statement of Theorem~\ref{thm:10} may be improved.
\begin{theorem}\label{thm:00}
Let~$n\geq1$, $u_1=0$ and~$p>\tilde p$ in~\eqref{eq:CP}, with~$\tilde p$ as in~\eqref{eq:pcrit0}. Moreover, let~$p<1+2/(n-2)$ if~$n\geq3$. Then there exists~$\eps>0$ such that for any~$u_0\in L^1\cap L^p$, satisfying~\eqref{eq:u0eps}, there exists a unique global solution as in~\eqref{eq:solution0}, to~\eqref{eq:CP}. Moreover, the solution satisfies the following estimates:
\begin{equation}
\label{eq:decayu0}
\|u(t,\cdot)\|_{L^q} \leq C\,(1+t)^{\alpha-\frac{n}2(1+\alpha)\left(1-\frac1q\right)}\,\|u_0\|_{L^1\cap L^p}, \qquad \forall \, q\in[1,p], \ \forall t\geq0,
\end{equation}
where~$C>0$ does not depend on the data.
\end{theorem}
\begin{remark}\label{rem:decay0}
The decay rate in~\eqref{eq:decayu0} is the same of the decay rate of the linear problem, in~\eqref{eq:decayulin}, when~$K>0$ and~$u_1\equiv0$. However, the decay rate in~\eqref{eq:decayu0} is worse than the decay rate of the homogeneous linear problem, which corresponds to take~$K=0$, when~$u_1\equiv0$ (see Remark~\ref{rem:decayhom}).
\end{remark}
\begin{remark}
The critical exponents~$\tilde p$ and~$\bar p$ verify
\[ 1+\frac2{n}<\tilde p < 1+\frac2{n-1}<\bar p < 1 +\frac2{n-2}. \]
We notice that~$\tilde p$ is increasing with respect to~$\alpha$, whereas~$\bar p$ is decreasing with respect to~$\alpha$.
\end{remark}
\begin{remark}
If~$\alpha\to0$, then~$\tilde p$ in~\eqref{eq:pcrit0} tends to Fujita exponent~$1+2/n$, the critical exponent for the semilinear heat equation. On the other hand, $\bar p$ tends to~$1+2/(n-2)$ as~$\alpha\to0$. This latter fact is less surprising than what may appear. Indeed, if~$u$ solves the linear Cauchy problem for the heat equation, then
\[ u_t(0,x)=\Delta u_0(x). \]
In particular, having in mind the properties of Riesz potential, assuming $u_t(0,\cdot)\in L^1\cap L^\infty$, only implies, that~$u_0\in L^r\cap L^\infty$, for any~$r>n/(n-2)$, in~\eqref{eq:Fujita} in space dimension~$n\geq3$. The limit of our exponent~$\bar p$ is then justified noticing that global solutions to~\eqref{eq:Fujita} exist when~$u_0\in L^r\cap L^\infty$ is small, for any~$r>n/(n-2)$, if~$p>1+2r/n=1+2/(n-2)$.
\end{remark}
\begin{remark}
As~$\alpha\to 1$, both the critical exponents~$\tilde p$ and~$\bar p$ tend to the exponent~$1+2/(n-1)$ found by Kato~\cite{K}. However, this latter is different from the critical exponent in~\eqref{eq:strauss} for the semilinear wave equation. The reason for this ``lack of continuity at~$\alpha=1$'' is that the influence of oscillations is neglected in the kernels for the fractional diffusive equation, for any~$\alpha\in(0,1)$, whereas it becomes relevant for the wave equation. The critical exponent~$1+2/(n-1)$ is the same found for the semilinear wave equation with structural damping~$(-\Delta)^{\frac12}u_t$ (see \cite{DA14proc, DARMMAS, DKR}). Indeed, the influence of oscillations is also neglected for this latter model, due to the presence of this special structural damping term, even if no diffusion phenomenon comes into play.
\end{remark}

\subsection{Discussion about the critical exponents in Theorems~\ref{thm:10} and~\ref{thm:00}}\label{sec:scaling}

Quite often, critical exponents for semilinear equations may be found by using scaling arguments. If~$\lambda\in(0,+\infty)$, then
\[ \partial_t^{1+\alpha} (f(\lambda t))= \lambda^{1+\alpha}\,(\partial_t^{1+\alpha} f)(\lambda t). \]
Therefore, given a solution~$u$ to the equation in~\eqref{eq:CP}, the function~$\lambda^{\frac{2}{p-1}}\,u(\lambda^{\frac2{1+\alpha}}t,\lambda u)$ is a solution to~\eqref{eq:CP} for any~$\lambda\in(0,+\infty)$. Due to
\[ \partial_t \bigl( u(\lambda^{\frac2{1+\alpha}}t,\lambda x)\bigr)\bigl|_{t=0}\bigr. = \lambda^{\frac2{1+\alpha}}\,u_1(\lambda x), \]
and
\[ \|\lambda^{\frac{2}{p-1}+\frac2{1+\alpha}}\,u_1(\lambda\cdot)\|_{L^q} = \lambda^{\frac{2}{p-1}+\frac2{1+\alpha}-\frac{n}q}\,\|u_1\|_{L^q}, \]
the scaling exponent for~\eqref{eq:CP} is
\[ q_\sca = \frac{n(p-1)}2\,\frac{1+\alpha}{p+\alpha}. \]
Indeed, as one hopefully expects, our critical exponent~$\bar p$ in~\eqref{eq:pcrit}, obtained for non-trivial~$u_1$, is the solution to~$q_\sca=1$.

When~$u_1$ is zero, we may try to apply the scaling arguments to~$u_0$, and see if the critical exponent~$\tilde p$ in~\eqref{eq:pcrit0} comes out, but this is not the case. Indeed, due to
\[ \|\lambda^{\frac{2}{p-1}}\,u_0(\lambda\cdot)\|_{L^q} = \lambda^{\frac{2}{p-1}-\frac{n}q}\,\|u_0\|_{L^q}, \]
the scaling exponent is~$q_\sca=n(p-1)/2$, the same of the heat equation. The motivation for this apparent inconsistency is that a loss of decay rate~$(1+t)^\alpha$ appears for the solution to~\eqref{eq:CP} with~$u_1\equiv0$, with respect to the homogeneous problem with~$u_1\equiv0$ (Theorem~\ref{thm:linear} with~$u_1\equiv0$ and~$f\equiv0$, see Remark~\ref{rem:decay0}). Indeed, for any~$q\in[1,\infty]$ if~$n=1$, and~$q<1+2/(n-2)$ if~$n\geq2$, the decay rate for the solution to~\eqref{eq:CP} in Theorem~\ref{thm:00} is given by
\[ \|u(t,\cdot)\|_{L^q} \lesssim (1+t)^{\alpha-\frac{n}2(1+\alpha)\left(1-\frac1q\right)}\,\|u_0\|_{L^1\cap L^p}.\]
The effect of having a critical exponent different from the solution to~$q_\sca=1$, as related to the presence of fractional integration, has been already observed for the heat equation with nonlinear memory~\cite{CDW}, namely, for
\begin{equation}\label{eq:heat}
\begin{cases}
\displaystyle
\partial_t u -\triangle u = \frac1{\Gamma(\alpha)}\,\int_0^t (t-s)^{-(1-\alpha)}\,|u(s,x)|^p\,ds \,,\\
u(0,x)=u_0(x)\,.
\end{cases}
\end{equation}
In this case, the critical exponent is
\begin{equation}\label{eq:pcritheat}
\max \left\{\hat p(n,\alpha)\,,\ \frac1{1-\alpha}\right\},\qquad \hat p(n,\alpha)\doteq 1+\frac{2(1+\alpha)}{n-2\alpha}.
\end{equation}
In particular, small data global solutions exist for~$p>\max\{\hat p(n,\alpha),1/(1-\alpha)\}$, and any solution blows up in finite time if~$1<p\leq\max\{\hat p(n,\alpha),1/(1-\alpha)\}$, provided that~$u_0\geq0$ is non-trivial. The same critical exponent remains valid for damped waves with nonlinear memory~\cite{DA14}.

We notice that~$\hat p(n,\alpha)>\tilde p(n,\alpha)$ for any~$\alpha\in(0,1)$. Indeed, one has~$\tilde p(n,\alpha)=\hat p (n(1+\alpha),\alpha)$. The relation between problem~\eqref{eq:CP} with~$u_1\equiv0$ and problem~\eqref{eq:heat}, and their critical exponents, becomes more clear in view of Remark~\ref{rem:CP0} in Section~\ref{sec:caputo}.

\subsection{Caputo and Riemann-Liouville fractional derivatives}\label{sec:caputo}

Caputo fractional derivative~\eqref{eq:caputo} may be written by means of the Riemann-Liouville fractional derivative, using the following relation (Theorem~$2.1$ in~\cite{KST}):
\[ \partial_t^{j+\alpha} g (t) = D^{j+\alpha} g_j(t), \qquad g_j(s)=g(s)- \sum_{k=0}^{j} \frac{g^{(k)}(0)}{k!}\,s^k, \]
where
\begin{equation}\label{eq:RLder}
D^{j+\alpha} h (t) = \partial_t^{j+1} (J^{1-\alpha} h)(t),
\end{equation}
for~$j\in\N$ and~$\alpha\in(0,1)$, is the Riemann-Liouville fractional derivative. Under suitable assumptions on~$g$, some Caputo fractional derivatives commute with ordinary derivatives, when applied to~$g$.
\begin{remark}\label{rem:0}
Let~$j\in\N\setminus\{0\}$ and assume~$g^{(k)}(0)=0$, for any~$k=1,\ldots,j$. Then
\begin{equation}\label{eq:CaputoRL}
\partial_t^{j+\alpha} g = \partial_t^j (\partial_t^\alpha g),
\end{equation}
for any~$\alpha\in(0,1)$. %, that is, $g$ is in the kernel of the commutator $[\partial_t^j,\partial_t^\alpha]$.
Indeed, due to
\[ g_j(t) = g(t)-g(0) = g_0(t), \]
we get
\[ \partial_t^{j+\alpha} g (t) = D^{j+\alpha} g_j (t) = D^{j+\alpha} g_0 (t) = \partial_t^j (D^\alpha g_0) (t) = \partial_t^j (\partial_t^\alpha g). \]
\end{remark}
We notice that
\begin{equation}\label{eq:CaputoRLtrick}
\partial_t^j (\partial_t^\alpha g)=D^{j-1+\alpha} (g').
\end{equation}
Between fractional integration and fractional differentiation it holds the following relation (see Lemma~$2.4$ in~\cite{KST}):
\begin{equation}
\label{eq:id+}
D^\alpha\, J^\alpha f = f,
\end{equation}
for any~$f\in L^p([0,T])$, for some~$p\in[1,\infty]$.
\begin{remark}\label{rem:CP0}
As a consequence of Remark~\ref{rem:0} and~\eqref{eq:CaputoRLtrick}, \eqref{eq:id+}, any solution to the following integro-differential problem:
\begin{equation}\label{eq:CP0}
\begin{cases}
u_t = J^{\alpha} \bigl( \Delta u + |u|^p\bigr), \qquad t>0,\ x\in\R^n,\\
u(0,x)=u_0(x),
\end{cases}
\end{equation}
also solves the Cauchy problem~\eqref{eq:CP} with~$u_1=0$. Indeed, applying~$D^\alpha$ to both sides of the equation in~\eqref{eq:CP0}, one obtains the equation in~\eqref{eq:CP}, and evaluating the equation in~\eqref{eq:CP0} at~$t=0$, one gets~$u_t(0,x)=0$. Problem~\eqref{eq:CP0} has been recently studied in~\cite{AV, DAEPproc}. In space dimension~$n=1$ it was first studied by Y. Fujita~\cite{F}.
\end{remark}

\subsection{Representation of the solution to the linear problem}

The following result for the Cauchy problem for Caputo fractional differential equations allows us to study~\eqref{eq:CPlinear} by using the Fourier transform with respect to~$x$.
\begin{theorem}\label{thm:KST}
[Theorem 4.3, Example 4.10 in~\cite{KST}]
Let~$\alpha\in(0,1)$, $b_{0},b_{1},\lambda\in\R$. Then the unique solution to
\[ \begin{cases}
\partial_t^{1+\alpha} g = \lambda g + f(t) & t>0,  \\
g(0)=b_0, \\
g'(0)=b_1,
\end{cases} \]
is given by
\begin{equation}\label{eq:CPsolution}
g(t)=b_0\,E_{1+\alpha,1}\bigl(\lambda\,t^{1+\alpha}\bigr)+b_1\,t\,E_{1+\alpha,2}\bigl(\lambda\,t^{1+\alpha}\bigr) + \int_0^t (t-s)^{\alpha}\,E_{1+\alpha,1+\alpha}\bigl(\lambda\,(t-s)^{1+\alpha}\bigr)\,f(s)\,ds,
\end{equation}
where~$E_{1+\alpha,\beta}$ are the Mittag-Leffler functions:
\[ E_{1+\alpha,\beta}(z) = \sum_{k=0}^\infty \frac{z^k}{\Gamma(k+\alpha k+\beta)}. \]
\end{theorem}
To manage the Mittag-Leffler functions in Theorem~\ref{thm:KST}, we will use the following representation.
\begin{theorem}\label{thm:ML}[Theorem~1.1.3 in~\cite{PS}]
Let~$\rho\in(1/3,1)$, $\beta\in\R$, and~$m\in\N$, with~$m\geq\rho\beta-1$. Then, for any~$z>0$ it holds:
\[ E_{1/\rho,\beta}(-z^{1/\rho})=2\rho\,z^{1-\beta}\,e^{z\cos(\pi\rho)}\,\cos(z\sin(\pi\rho)-\pi\rho(\beta-1))+\sum_{k=1}^m \frac{(-1)^{k-1}}{\Gamma(\beta-k/\rho)}\,z^{-k/\rho} +\Omega_m,\]
where
\[ \Omega_m(z)=\frac{(-1)^m\,z^{1-\beta}}\pi \Bigl( I_{1,m}\,\sin\bigl(\pi(\beta-(m+1)/\rho)\bigr) + I_{2,m}\,\sin\bigl(\pi(\beta-m/\rho)\bigr)\Bigr), \]
and
\[ I_{j,m}(z) = \int_0^\infty \frac{s^{(m+j)/\rho-\beta}}{s^{2/\rho}+2\cos(\pi/\rho)\,s^{1/\rho}+1}\,e^{-zs}\,ds. \]
\end{theorem}
\begin{Not}
In Theorem~\ref{thm:ML} and in all the paper, we will use the notation~$\rho=1/(1+\alpha)$. In particular, $\rho\in(1/2,1)$, due to~$\alpha\in(0,1)$.
\end{Not}
\begin{remark}\label{rem:uniformI}
We notice that~$I_{j,m}(z)$ is uniformly bounded with respect to~$z\in(0,\infty)$, that is,
\[ \int_0^\infty \frac{s^{(m+j)/\rho-\beta}}{s^{2/\rho}+2\cos(\pi/\rho)\,s^{1/\rho}+1}\,ds<\infty, \]
if, and only if,
\begin{equation}\label{eq:uniformI}
-1 < m+j-1+\rho(1-\beta) < 1.
\end{equation}
In particular, condition~\eqref{eq:uniformI} holds for~$\beta=1$ and~$m+j=1$, and for~$\beta=1/\rho$ and~$m+j=2$. On the other hand, it does not hold for~$\beta=2$ and~$m+j=3$.
\end{remark}

%%%%%%%%%%%%%%%%%%%%%%%%%%%%%%%%

\section{Linear estimates}\label{sec:linear}

We first consider linear problem~\eqref{eq:CPlinear}. After performing the Fourier transform with respect to~$x$, $\hat u=\mathfrak{F}_x u(t,\xi)$, we obtain
\begin{equation}\label{eq:CPF}
\begin{cases}
\partial_t^{1+\alpha} \hat u + \xii^{2} \hat u = \hat f(t,\xi),\qquad t>0,\ x\in\R^n,\\
\hat u(0,\xi)=\hat u_0(\xi),\\
\hat u_t(0,\xi)=\hat u_1(\xi).
\end{cases}
\end{equation}
Thanks to Theorem~\ref{thm:KST}, the solution to~\eqref{eq:CPlinear} is given by
\begin{equation}\label{eq:solutiondef}
u(t,\cdot)=u^\hom(t,\cdot) + \int_0^t (t-s)^{\alpha}\,G_{1+\alpha,1+\alpha}(t-s,\cdot) \ast_{(x)} f((s,\cdot))\,ds,
\end{equation}
where the homogeneous part of the solution is given by (see also~\cite{VK06})
\begin{equation}\label{eq:uhom}
u^\hom(t,\cdot)=G_{1+\alpha,1}(t,\cdot) \ast_{(x)} u_0 +t\,G_{1+\alpha,2}(t,\cdot) \ast_{(x)} u_1,
\end{equation}
and
\begin{equation}\label{eq:Gdef}
G_{1+\alpha,\beta}(t,x) = \mathfrak{F}^{-1} \Bigl( E_{1+\alpha,\beta}\bigl(-t^{1+\alpha}\,\xii^{2}\bigr) \Bigr).
\end{equation}
We will estimate the $L^r-L^q$ mapping properties, $1\leq r\leq q\leq\infty$, of the operator~$G_{1+\alpha,\beta}(t,x)\ast_{(x)}$, applying Theorem~\ref{thm:ML} with~$\rho=1/(1+\alpha)$ and~$z=t\xii^{2\rho}$. We directly estimate the fundamental solution~$G_{1/\rho,\beta}(1,\cdot)$ in~$L^p$ norms, having in mind Young inequality. By virtue of the scaling property
\begin{equation}\label{eq:scalingt}
\bigl\|\mathfrak{F}^{-1}\bigl(m(t^{\frac1{\rho}}|\cdot|^2)\bigr)\bigr\|_{L^p} = t^{-\frac{n}{2\rho}\left(1-\frac1p\right)}\bigl\|\mathfrak{F}^{-1}\bigl(m(|\cdot|^2)\bigr)\bigr\|_{L^p},
\end{equation}
it will be sufficient to consider~$G_{1/\rho,\beta}(1,\cdot)$. We will distinguish the three cases, $\beta=1,1/\rho,2$.

\subsection{Estimate for~$G_{1/\rho,1}$}

Having in mind the representation in Theorem~\ref{thm:ML}, we define
\begin{align}
\label{eq:Kdef}
K_{1/\rho,d}
    & = \mathfrak{F}^{-1} \bigl( \xii^d\,e^{\xii^{2\rho}\cos(\pi\rho)}\,\cos(\xii^{2\rho}\sin(\pi\rho)) \bigr),\\
\label{eq:Hdef}
H_{1/\rho,d}(s,\cdot)
    & =\mathfrak{F}^{-1}\bigl(\xii^d\,e^{-s\xii^{2\rho}}\bigr),\quad \forall s>0,
\end{align}
where~$d\in\R$. We first consider~\eqref{eq:Hdef}. By scaling property~\eqref{eq:scalingt}, we get
\begin{equation}\label{eq:scalings}
\| H_{1/\rho,d}(s,\cdot) \|_{L^p} = s^{-\frac{n}{2\rho}\left(1-\frac1p\right)-\frac{d}{2\rho}} \| H_{1/\rho,d}(1,\cdot) \|_{L^p},
\end{equation}
for any~$p\in[1,\infty]$, so it is sufficient to study in which cases the right-hand side is finite.

\begin{lemma}\label{lem:H}
Let~$\rho>0$ and~$d>-n$. Then we may distinguish two cases.
\begin{itemize}
\item If~$d\geq0$, then~$H_{1/\rho,d}(1,\cdot) \in L^p$, for any~$p\in[1,\infty]$.
\item If~$d\in(-n,0)$, then~$H_{1/\rho,d}(1,\cdot) \in L^p$, for any~$p\in(1,\infty]$, such that
\begin{equation}\label{eq:Lpd}
n\left(1-\frac1p\right)>-d.
\end{equation}
\end{itemize}
\end{lemma}
\begin{remark}
The result in Lemma~\ref{lem:H} is probably well-known to the reader, possibly with slight modifications and/or in a more general formulation. However, the scheme we used to prove it, will appear again, later in this paper, so we provide the reader with a self-contained proof of Lemma~\ref{lem:H}.
\end{remark}
\begin{proof}
It is clear that~$H_{1/\rho,d}(1,\cdot) \in L^\infty$ by Riemann-Lebesgue theorem, since~$\xii^{d}\,e^{-\xii^{2\rho}}$ is in~$L^1$, for any~$d>-n$.

First, let~$d=0$. It holds (see~\cite{BG})
\[ \lim_{|x|\to\infty} |x|^{n+2\rho}\,\mathfrak{F}^{-1} (e^{-\xii^{2\rho}})(x) = C_{n,\rho}, \]
for any~$\rho>0$, so that~$H_{1/\rho,0}(1,\cdot) \in L^1$, as well. Therefore, $H_{1/\rho,0}(1,\cdot) \in L^p$, for any~$p\in[1,\infty]$.

Now, let~$d\in(-n,0)$, and~$p\in(1,\infty)$, verifying~\eqref{eq:Lpd}. Setting~$p^*\in(1,-n/d)$ as
\[ \frac1{p^*}-\frac1p = \frac{-d}n, \]
we get
\[ \| H_{1/\rho,d}(1,\cdot) \|_{L^p} = \| (-\Delta)^{\frac{d}2}\,H_{1/\rho,0}(1,\cdot) \|_{L^p} \lesssim \|H_{1/\rho,0}(1,\cdot) \|_{L^{p^*}}, \]
by Riesz potential mapping properties. Therefore, $H_{1/\rho,d}(1,\cdot) \in L^p$.

Finally, let~$d>0$ and~$p=1$. By using the property
\[ e^{ix\xi} = \sum_{j=1}^n \frac{(-ix_j)}{|x|^2}\partial_{\xi_j} e^{ix\xi}, \]
and integrating by parts, we may write
\[ H_{1/\rho,d}(1,x)=|x|^{-k}\,(2\pi)^{-n}\sum_{|\gamma|=k}\left(\frac{ix}{|x|}\right)^\gamma \int_{\R^n} e^{ix\xi}\,\partial_\xi^\gamma\left(\xii^{d}\,e^{-\xii^{2\rho}}\right)\,d\xi, \]
for any~$k\in\N$. We notice that, in general,
\begin{equation}\label{eq:expcompensate}
|\partial_\xi^\gamma\left(\xii^{d}\,e^{-\xii^{2\rho}}\right)|\lesssim \xii^{d-|\gamma|}\,(1+\xii^{2\rho})^{|\gamma|}\,e^{-\xii^{2\rho}}\lesssim \xii^{d-|\gamma|}\,e^{-c\xii^{2\rho}},
\end{equation}
for some~$c\in(0,1)$.

\bigskip

If~$d>1$, then, taking~$k=n+1$, we may trivially estimate
\[ |H_{1/\rho,d}(1,x)|\lesssim |x|^{-(n+1)}\,\int_{\R^n} \xii^{d-(n+1)}\,e^{-c\xii^{2\rho}}d\xi \lesssim |x|^{-(n+1)}. \]
If~$d\in(0,1]$, we proceed in a different way. Let~$\gamma\in\N^n$, with~$|\gamma|=n$. We may split each integral into two parts:
\begin{multline*}
\int_{\R^n} e^{ix\xi}\,\partial_\xi^\gamma\left(\xii^d\,e^{-\xii^{2\rho}}\right)\,d\xi= I_0(x)+I_1(x) \\
= \int_{\xii\leq |x|^{-1}}e^{ix\xi}\,\partial_\xi^\gamma\left(\xii^d\,e^{-\xii^{2\rho}}\right)\,d\xi + \int_{\xii\geq |x|^{-1}}e^{ix\xi}\,\partial_\xi^\gamma\left(\xii^d\,e^{-\xii^{2\rho}}\right)\,d\xi.
\end{multline*}
On the one hand, we trivially estimate
\[ |I_0(x)| \lesssim \int_{\xii\leq |x|^{-1}} \xii^{d-n}\,e^{-c\xii^{2\rho}}\,d\xi\lesssim \int_{\xii\leq |x|^{-1}} \xii^{d-n}\,d\xi \lesssim |x|^{-d}. \]
On the other hand, we perform one additional step of integration by parts in~$I_1$. If~$d\in(0,1)$, we obtain
\[ |I_1(x)| \lesssim |x|^{-1}\,\int_{\xii=|x|^{-1}} \xii^{d-n}\,d\sigma + |x|^{-1}\,\int_{\xii\geq |x|^{-1}}\,\xii^{d-(n+1)}\,d\xi\lesssim |x|^{-d},\]
whereas, if~$d=1$, we split each integral into two parts (for large~$|x|$):
\[ \int_{\R^n} e^{ix\xi}\,\partial_{\xi_j}\partial_\xi^\gamma\left(\xii\,e^{-\xii^{2\rho}}\right)\,d\xi= I_{1,1}(x)+I_{1,2}(x)= \int_{|x|^{-1}\leq \xii\leq1}\ldots d\xi + \int_{\xii\geq 1}\ldots d\xi,\]
directly estimating~$I_{1,1}$, and performing one additional step of integration by parts in~$I_{1,2}$. This leads to
\[ |I_{1,1}(x)|\lesssim \log (1+|x|), \qquad |I_{1,2}(x)|\leq C.\]
Summarizing, we proved that
\[ |H_{1/\rho,d}(1,x)|\lesssim
\begin{cases}
|x|^{-(n+1)} & \text{if~$d>1$,}\\
|x|^{-(n+1)}\log(1+|x|) & \text{if~$d=1$,}\\
|x|^{-(n+d)} & \text{if~$d\in(0,1)$.}
\end{cases} \]
Recalling that $H_{1/\rho,d}(1,\cdot)\in L^\infty$, we obtained that $H_{1/\rho,d}(1,\cdot)\in L^1$ as well. Therefore, $H_{1/\rho,d}(1,\cdot) \in L^p$, for any~$p\in[1,\infty]$.
\end{proof}
\begin{remark}\label{rem:extrarho}
An additional power~$\xii^{2\rho}$ appears in~\eqref{eq:expcompensate}, when~$d$ is positive and even, and~$|\gamma|\geq d+1$, so that at least one derivative is applied to the exponential term, namely, one has
\[ |\partial_\xi^\gamma\left(\xii^{d}\,e^{-\xii^{2\rho}}\right)|\lesssim \xii^{2\rho+d-|\gamma|}\,(1+\xii^{2\rho})^{|\gamma|-1}\,e^{-\xii^{2\rho}}\lesssim \xii^{2\rho+d-|\gamma|}\,e^{-c\xii^{2\rho}}, \]
However, this improvement in this special case, is not necessary in the proof of Lemma~\ref{lem:H} for~$d>0$.
\end{remark}
With the scheme used for~$H_{1/\rho,d}(1,\cdot)$, we may deal with~$K_{1/d,\rho}$, defined in~\eqref{eq:Kdef}.
\begin{lemma}\label{lem:exp}
Let~$\rho\in(1/2,1)$ and~$d>-n$. Then we may distinguish two cases.
\begin{itemize}
\item If~$d\geq0$, then~$K_{1/\rho,d} \in L^p$, for any~$p\in[1,\infty]$.
\item If~$d\in(-n,0)$, then~$K_{1/\rho,d}\in L^p$, for any~$p\in(1,\infty]$, such that~\eqref{eq:Lpd} holds.
\end{itemize}
\end{lemma}
\begin{proof}
We notice that $\cos(\pi\rho)<0$, due to~$\rho\in(1/2,1)$. It is clear that~$K\in L^\infty$, since~$\hat K\in L^1$, due to~$d>-n$.

First, let us consider the special case~$d=0$. Performing the integration by parts as in the proof of Lemma~\ref{lem:H}, we obtain~$|K(x)|\lesssim |x|^{-(n+1)}$ for large~$|x|$, due to~$2\rho>1$. Indeed, after the first step of integration by parts, we have:
\begin{align*}
& \partial_{\xi_j} \bigl(e^{\xii^{2\rho}\cos(\pi\rho)}\,\cos(\xii^{2\rho}\sin(\pi\rho))\bigr)\\
& \qquad = 2\rho\xii^{2\rho-1}\,\frac{\xi_j}{\xii}\,e^{\xii^{2\rho}\cos(\pi\rho)}\, \left(\cos(\pi\rho)\,\cos(\xii^{2\rho}\sin(\pi\rho))-\sin(\pi\rho)\,\sin(\xii^{2\rho}\sin(\pi\rho))\right)\\
& \qquad = 2\rho\xii^{2\rho-1}\,\frac{\xi_j}{\xii}\,e^{\xii^{2\rho}\cos(\pi\rho)}\,\cos(\pi\rho+\xii^{2\rho}\sin(\pi\rho)),
\end{align*}
so that
\begin{equation}\label{eq:goodestimate}
\bigl|\partial_\xi^\gamma \bigl(e^{\xii^{2\rho}\cos(\pi\rho)}\,\cos(\xii^{2\rho}\sin(\pi\rho))\bigr)\bigr| \lesssim \xii^{2\rho-|\gamma|}\,e^{-c\xii^{2\rho}},
\end{equation}
for some~$c\in(0,-\cos(\pi\rho))$, for any~$|\gamma|\geq1$. We may now follow the steps of the proof of Lemma~\ref{lem:H} for~$d>1$, thanks to the presence of the term~$2\rho>1$. (Incidentally, we notice that the argument may be also refined to prove that~$|K_{1/\rho,0}(x)|\lesssim |x|^{-(n+2\rho)}$ for large~$|x|$, see Remark~\ref{rem:extrarho}). % and, with more easy calculations, that~$|H_{1/\rho,0}(1,x)|\lesssim |x|^{-(n+2\rho)}$ for large~$|x|$.

\bigskip

For~$d\in(-n,0)$ and~$d>0$, the proof is analogous to the proof of Lemma~\ref{lem:H} for~$H_{1/\rho,d}(1,x)$. We notice that, for~$d\neq0$, we only have, in general,
\[ \bigl|\partial_\xi^\gamma \bigl(\xii^d \, e^{\xii^{2\rho}\cos(\pi\rho)}\,\cos(\xii^{2\rho}\sin(\pi\rho))\bigr)\bigr| \lesssim \xii^{d-|\gamma|}\,e^{-c\xii^{2\rho}}, \]
as in~\eqref{eq:expcompensate}, instead of estimate~\eqref{eq:goodestimate}.
\end{proof}
Recalling that~$\rho=1/(1+\alpha)\in(1/2,1)$, we are now ready to estimate~$G_{1/\rho,1}(1,\cdot)$.
\begin{proposition}\label{prop:E1}
For any~$\rho\in(1/2,1)$, it holds
\[ G_{1/\rho,1}(1,\cdot)\in L^p, \]
for any~$p\in[1,\infty]$ such that
\begin{equation}\label{eq:boundLp}
\frac{n}2\left(1-\frac1p\right)<1.
\end{equation}
\end{proposition}
\begin{proof}
According to Theorem~\ref{thm:ML},
\[ G_{1/\rho,1}(1,\cdot) = 2\rho\,K_{1/\rho,0} + \pi^{-1}\,\sin(\pi(1-1/\rho))\,\int_0^\infty \frac{s^{1/\rho-1}}{s^{2/\rho}+2\cos(\pi/\rho)\,s^{1/\rho}+1}\,H_{1/\rho,0}(s,x)\,ds, \]
taking~$m=\ceil{\rho-1}=0$, and~$\sin\pi=0$. Therefore, having in mind~\eqref{eq:scalings}, the proof follows from Lemmas~\ref{lem:H} and~\ref{lem:exp} if the integral
\[ \int_0^\infty \frac{s^{1/\rho-1-\frac{n}{2\rho}\left(1-\frac1p\right)}\,}{s^{2/\rho}+2\cos(\pi/\rho)\,s^{1/\rho}+1}\,ds \]
converges, that is, if
\[-1<\frac{n}2\left(1-\frac1p\right)<1.\]
This concludes the proof.
\end{proof}

\subsection{Estimate for~$G_{1/\rho,1/\rho}$}

Proceeding as in Proposition~\ref{prop:E1}, we have the following preliminary result for~$G_{1/\rho,1/\rho}$.
\begin{lemma}\label{lem:E1alphaeasy}
For any~$\rho\in(1/2,1)$, it holds
\[ G_{1/\rho,1/\rho}(1,\cdot)\in L^p, \]
for any~$p\in(1,\infty]$ such that
\begin{equation}\label{eq:preboundLpalpha}
1-\rho < \frac{n}{2}\left(1-\frac1p\right)<2.
\end{equation}
\end{lemma}
\begin{proof}
According to Theorem~\ref{thm:ML},
\[ G_{1/\rho,1/\rho}(1,\cdot) = 2\rho\,K_{1/\rho,-2(1-\rho)} + \pi^{-1}\,\sin(\pi/\rho)\,\int_0^\infty \frac{s^{1/\rho}}{s^{2/\rho}+2\cos(\pi/\rho)\,s^{1/\rho}+1}\,H_{1/\rho,-2(1-\rho)}(s,x)\,ds, \]
due to~$m=0$, and~$\sin0=0$. We may apply Lemmas~\ref{lem:H} and~\ref{lem:exp} if~\eqref{eq:Lpd} holds with~$d=-2(1-\rho)$, i.e., if
\begin{equation}\label{eq:preboundLpalphabelow}
\frac{n}2\left(1-\frac1p\right)>1-\rho,
\end{equation}
which is guaranteed by the left-hand side of~\eqref{eq:preboundLpalpha}.

Therefore, having in mind~\eqref{eq:scalings}, the proof follows from Lemmas~\ref{lem:H} and~\ref{lem:exp} if the integral
\[ \int_0^\infty \frac{s^{2/\rho-1-\frac{n}{2\rho}\left(1-\frac1p\right)}\,}{s^{2/\rho}+2\cos(\pi/\rho)\,s^{1/\rho}+1}\,ds \]
converges, that is, if
\[ 0<\frac{n}2\left(1-\frac1p\right)<2, \]
which is guaranteed by the right-hand side of~\eqref{eq:preboundLpalpha}. This concludes the proof.
\end{proof}
In order to relax bound~\eqref{eq:preboundLpalphabelow}, we may modify our approach, relying on the use of the representation in Theorem~\ref{thm:ML} only at large frequencies.
\begin{remark}\label{rem:entire}
The function~$E_{1/\rho,\beta}(-\xii^2)$ is in~$\mathcal C^\infty$ (as a complex-valued function, $E_{1/\rho,\beta}(z)$ is entire). Let~$\chi$ be a~$\mathcal C^\infty$ radial function, vanishing in the ball~$\{\xii\leq1\}$, satisfying~$0\leq\chi\leq1$, and~$\chi=1$ out of some compact set. Then~$(1-\chi(\xi))E_{1/\rho,\beta}$ is in the Schwartz space~$\mathcal S$, in particular
\[ \mathfrak{F}^{-1} \bigl((1-\chi(\xi))E_{1/\rho,\beta}(-\xii^2)\bigr) \]
is in~$L^1\cap L^\infty$. We define
\begin{align}
\label{eq:Ktildedef}
\tilde K_{1/\rho,d}
    & =\mathfrak{F}^{-1}\bigl(\chi\,\hat K_{1/\rho,d}\bigr), \\
\label{eq:Htildedef}
\tilde H_{1/\rho,d}(s,\cdot)
    & =\mathfrak{F}^{-1}\bigl(\chi\,\hat H_{1/\rho,d}(s,\cdot)\bigr),\quad \forall s>0.
\end{align}
It is clear that~$\tilde K_{1/\rho,d} \in L^1\cap L^\infty$ for any~$d\in\R$, since~$\chi\,\hat K_{1/\rho,d}\in\mathcal S$. Similarly, $\tilde H_{1/\rho,d}(s,\cdot) \in L^1\cap L^\infty$ for any~$d\in\R$ and~$s>0$, but its norm depends, in general, on~$s$.
\end{remark}
In view of Remark~\ref{rem:entire}, we consider~$\tilde H_{1/\rho,d}(s,\cdot)$. We first consider the range of exponents~$p$ for which the~$L^p$ norm of~$\tilde H_{1/\rho,d}(s,\cdot)$ is uniformly bounded, with respect to~$s$.
\begin{lemma}\label{lem:Htilde}
Let~$\rho>0$ and~$d<0$. Let~$p\in[1,\infty]$, be such that
\begin{equation}\label{eq:Lpdback}
n\left(1-\frac1p\right)<-d.
\end{equation}
Then
\[ \|\tilde H_{1/\rho,d}(s,\cdot)\|_{L^p} \leq C, \]
uniformly with respect to~$s$.
\end{lemma}
\begin{proof}
We use the scheme of integration by parts in the proof of Lemma~\ref{lem:H}. Due to
\[ |\partial_{\xi}^\gamma e^{-s\xii^{2\rho}}| \leq C \xii^{-|\gamma|}, \]
with~$C$ independent of~$s$, we may estimate
\[ |\partial_\xi^\gamma(\chi(\xi)\xii^{d}\,e^{-s\xii^{2\rho}})| \lesssim \xii^{d-|\gamma|}, \]
for any~$\gamma$. For large~$|x|$, after~$n+1$ steps of integration by parts, we derive
\[ |\tilde H_{1/\rho,d}(s,x)| \lesssim |x|^{-(n+1)}\int_{\xii\geq1} \xii^{d-(n+1)}\,d\xi \lesssim |x|^{-(n+1)}. \]
If~$d<-n$, by Riemann-Lebesgue theorem, we directly obtain
\[ \|\tilde H_{1/\rho,d}(s,\cdot)\|_{L^\infty} \lesssim \int_{\xii\geq1} \xii^d\,d\xi \leq C, \]
and the thesis follows. Let~$d\in[-n,0)$; we set~$\kappa = \floor{n+d}$, that is, $\kappa\in(n+d-1,n+d]$, integer. For small~$|x|$, after $\kappa$ steps of integration by parts, we split each integral into two parts:
\begin{multline*}
\int_{\xii\geq1} e^{ix\xi}\,\partial_\xi^\gamma\left(\chi(\xi)\xii^d\,e^{-s\xii^{2\rho}}\right)\,d\xi= I_0(s,x)+I_1(s,x) \\
= \int_{1\leq\xii\leq |x|^{-1}}e^{ix\xi}\,\partial_\xi^\gamma\left(\chi(\xi)\xii^{d}\,e^{-s\xii^{2\rho}}\right)\,d\xi + \int_{\xii\geq |x|^{-1}}e^{ix\xi}\,\partial_\xi^\gamma\left(\chi(\xi)\xii^d\,e^{-s\xii^{2\rho}}\right)\,d\xi.
\end{multline*}
On the one hand,
\[ |I_0(s,x)| \lesssim \int_{1\leq\xii\leq |x|^{-1}} \xii^{d-\kappa}\,d\xi \lesssim \begin{cases}
|x|^{-(d+n-\kappa)} & \text{if~$d\in(\kappa-n,\kappa+1-n)$,}\\
-\log|x| & \text{if~$d=\kappa-n$.}
\end{cases} \]
On the other hand, performing one additional step of integration by parts in~$I_1$, we obtain
\[ |I_1(s,x)| \lesssim |x|^{-1}\,\int_{\xii=|x|^{-1}} \xii^{d-\kappa}\,d\sigma + |x|^{-1}\,\int_{\xii\geq |x|^{-1}}\,\xii^{d-\kappa-1}\,d\xi\lesssim |x|^{-(d+n-\kappa)}.\]
Summarizing, we obtained, for small~$|x|$,
\[ |\tilde H_{1/\rho,d}(s,x)|\lesssim \begin{cases}
|x|^{-(n+d)}(-\log|x|) & \text{if~$d$ is integer,}\\
|x|^{-(n+d)} & \text{if~$d$ is not integer.}
\end{cases} \]
Together with the previous estimate for large~$|x|$, this proves that
\[ \|\tilde H_{1/\rho,d}(s,\cdot) \|_{L^p}\leq C, \]
uniformly with respect to~$s>0$, for any~$p$ such that~\eqref{eq:Lpdback} holds.
\end{proof}
\begin{proposition}\label{prop:Ealpha}
For any~$\rho\in(1/2,1)$, it holds
\[ G_{1/\rho,1/\rho}(1,\cdot)\in L^p, \]
for any~$p\in[1,\infty]$ such that
\begin{equation}\label{eq:boundLpalpha}
\frac{n}2\left(1-\frac1p\right)<2.
\end{equation}
\end{proposition}
\begin{proof}
In view of Lemma~\ref{lem:E1alphaeasy}, $G_{1/\rho,1/\rho}(1,\cdot)\in L^p$ for some~$p$, verifying~\eqref{eq:preboundLpalpha}. Indeed, such~$p$ exists, due to~$1-\rho<1/2$. Therefore, it is sufficient to prove our statement for~$p=1$, so that, by interpolation, we conclude the proof.

Due to Remark~\ref{rem:entire} and to the representation in Theorem~\ref{thm:ML}
\begin{align*}
\mathfrak{F}^{-1}\bigl(\chi(\xi)\,E_{1/\rho,1/\rho}(1,-\xii^2)\bigr)
    & = 2\rho\,\tilde K_{1/\rho,-2(1-\rho)} \\
    & \qquad + \pi^{-1}\,\sin(\pi/\rho)\,\int_0^\infty \frac{s^{1/\rho}}{s^{2/\rho}+2\cos(\pi/\rho)\,s^{1/\rho}+1}\,\tilde H_{1/\rho,-2(1-\rho)}(s,\cdot)\,ds,
\end{align*}
the proof of our statement follows if
\[ \int_0^\infty \frac{s^{1/\rho}}{s^{2/\rho}+2\cos(\pi/\rho)\,s^{1/\rho}+1}\,\tilde H_{1/\rho,-2(1-\rho)}(s,x)\,ds \]
belongs to~$L^1$. Using Lemma~\ref{lem:Htilde} with~$p=1$, this latter property holds due to the convergence of the integral (see Remark~\ref{rem:uniformI})
\[ \int_0^\infty \frac{s^{1/\rho}}{s^{2/\rho}+2\cos(\pi/\rho)\,s^{1/\rho}+1}\,ds, \]
and this concludes the proof.
\end{proof}

\subsection{Estimate for~$G_{1/\rho,2}$}

The estimate for~$G_{1/\rho,2}$ is more difficult to be obtained, due to the fact that the representation of $E_{1/\rho,2}(-\xii^2)$ given by Theorem~\ref{thm:ML} contains a Riesz potential term~$(-\Delta)^{-1}$, taking~$m=\ceil{2\rho-1}=1$. For this reason, it is more convenient to rely on Remark~\ref{rem:entire} for any~$p$. Indeed, if we use the representation in Theorem~\ref{thm:ML} only for large~$\xii$, the Riesz potential behaves like a Bessel potential~$(1-\Delta)^{-1}$, whose mapping properties are better.

We conveniently modify Lemma~\ref{lem:H}.
\begin{lemma}\label{lem:Htilde2}
Let~$\rho\in(1/2,1)$ and~$d>-n$. Then we may distinguish two cases.
\begin{itemize}
\item If~$d>0$, then
\begin{equation}\label{eq:Htildes}
\|\tilde H_{1/\rho,d}(s,\cdot)\|_{L^p}\leq C\, s^{-\frac{d}{2\rho}-\frac{n}{2\rho}\left(1-\frac1p\right)},
\end{equation}
for any~$p\in[1,\infty]$, where~$C>0$ does not depend on~$s$.
\item If~$d\in(-n,0]$, then~\eqref{eq:Htildes} holds for any~$p\in(1,\infty]$, such that~\eqref{eq:Lpd} holds.
\end{itemize}
\end{lemma}
\begin{proof}
We may follow the proof of Lemma~\ref{lem:H} with two modifications. First of all, we do no longer have a special representation when~$d=0$, so that this case shall be treated together with~$d\in(-n,0)$. Then, after applying~\eqref{eq:scalingt}, we get
\[ \|\tilde H_{1/\rho,d}(s,\cdot)\|_{L^p}= s^{-\frac{d}{2\rho}-\frac{n}{2\rho}\left(1-\frac1p\right)} \|H_{1/\rho,d}^\dagger(s,\cdot)\|_{L^p},\]
where we defined
\[ H_{1/\rho,d}^\dagger(s,\cdot) = \mathfrak{F}^{-1}\bigl(\chi(s^{-\frac1{2\rho}}\xi)\,\hat H_{1/\rho,d}(1,\cdot)\bigr). \]
Therefore, we shall discuss what happens when derivatives are applied to~$\chi(s^{-\frac1{2\rho}}\xi)$. Due to the fact that
\[ \supp \partial_\xi^\gamma\chi \subset \{1 \leq \xii \leq K\}, \]
for any~$\gamma\neq0$, for some~$K>1$ and that~$\chi$ is smooth, we derive that:
\[ |\partial_\xi^\gamma \chi(s^{-\frac1{2\rho}}\xi)| \lesssim \begin{cases}
s^{-\frac{|\gamma|}{2\rho}} & \text{if~$s^{\frac1{2\rho}}\leq \xii\leq Ks^{\frac1{2\rho}}$,}\\
0 & \text{otherwise.}
\end{cases} \]
In particular, $|\partial_\xi^\gamma \chi(s^{-\frac1{2\rho}}\xi)|\lesssim \xii^{-|\gamma|}$. Now, we are able to follow the proof of Lemma~\ref{lem:H} to estimate~$H_{1/\rho,d}^\dagger(s,\cdot)$, uniformly with respect to~$s$.

It is clear that~$H_{1/\rho,d}^\dagger(s,\cdot) \in L^\infty$ by Riemann-Lebesgue theorem, and
\[ \|H_{1/\rho,d}^\dagger(s,\cdot)\|_{L^\infty} \lesssim \int_{\R^n} \xii^d\,e^{-\xii^{2\rho}}\,d\xi \leq C, \]
uniformly with respect to~$s$, for any~$d>-n$. For~$d>0$ and~$p=1$, we proceed as in the proof of Lemma~\ref{lem:H}, in particular, replacing~\eqref{eq:expcompensate} by
\begin{equation}\label{eq:expcompensatedagger}
\bigl|\partial_\xi^\gamma\bigl(\chi(s^{-\frac1{2\rho}}\xi)\xii^{d}\,e^{-\xii^{2\rho}}\bigr)\bigr|\lesssim \xii^{d-|\gamma|}\,(1+\xii^{2\rho})^{|\gamma|}\,e^{-\xii^{2\rho}}\lesssim \xii^{d-|\gamma|}\,e^{-c\xii^{2\rho}},
\end{equation}
for some~$c\in(0,1)$. Finally, for~$d\in(-n,0]$, and~$p\in(1,\infty)$, verifying~\eqref{eq:Lpd}, we use Riesz potential mapping properties, as in the proof of Lemma~\ref{lem:H}.
\end{proof}
Lemmas~\ref{lem:Htilde} and~\ref{lem:Htilde2} are valid for all~$s>0$, but, for large values of~$s$, Lemma~\ref{lem:Htilde} is not useful, since the integral
\[ \int_0^\infty \frac{s^{3/\rho-2}\,}{s^{2/\rho}+2\cos(\pi/\rho)\,s^{1/\rho}+1}\,ds, \]
does not converge (see Remark~\ref{rem:uniformI}).
\begin{lemma}\label{lem:Htilde3}
Let~$\rho\in(1/2,1)$ and~$d\in\R$. Then, for any~$s\geq1$, it holds
\begin{equation}\label{eq:Htildeexps}
\|\tilde H_{1/\rho,d}(s,\cdot)\|_{L^p}\leq C\,e^{-cs},
\end{equation}
for any~$c\in(0,1)$, for some~$C>0$, independent of~$s$, and for any~$p\in[1,\infty]$.
\end{lemma}
\begin{proof}
We use the scheme of integration by parts in the proof of Lemma~\ref{lem:H}, but now we may restrict to consider~$\xii\geq1$, due to the presence of~$\chi(\xi)$, and we may use the assumption~$s\geq1$ to avoid singularity at~$s=0$. Due to
\[ |\partial_{\xi}^\gamma e^{-s\xii^{2\rho}}| \leq C\,\xii^{-|\gamma|}\,e^{-cs\xii^{2\rho}}, \]
for any~$c\in(0,1)$, for some~$C>0$ independent of~$s$, we may estimate
\[ |\partial_\xi^\gamma(\chi(\xi)\xii^{d}\,e^{-s\xii^{2\rho}})| \lesssim \xii^{d-|\gamma|}\,e^{-cs\xii^{2\rho}}, \]
for any~$\gamma$. After~$k$ steps of integration by parts, we derive
\begin{align*}
|\tilde H_{1/\rho,d}(s,x)|
    & \lesssim |x|^{-k}\int_{\xii\geq1} \xii^{d-k}\,e^{-cs\xii^{2\rho}}\,d\xi \\
    & \lesssim |x|^{-k}\,e^{-c_1s}\,\int_{\xii\geq1} \xii^{d-k}\,e^{-(c-c_1)\xii^{2\rho}}\,d\xi\lesssim |x|^{-k}\,e^{-c_1s}.
\end{align*}
for any~$c_1\in(0,c)$, where we used~$s\geq1$. Therefore,
\[ \|\tilde H_{1/\rho,d}(s,\cdot) \|_{L^p}\leq C\,e^{-cs},\qquad \forall p\in[1,\infty]. \]
\end{proof}
\begin{proposition}\label{prop:E2}
For any~$\rho\in(1/2,1)$, it holds
\[ G_{1/\rho,2}(1,\cdot)\in L^p, \]
for any~$p\in[1,\infty]$ such that
\begin{equation}\label{eq:boundLp2}
\frac{n}{2}\left(1-\frac1p\right)<1.
\end{equation}
\end{proposition}
\begin{proof}
According to Theorem~\ref{thm:ML},
\begin{align*}
\mathfrak{F}^{-1} \bigl(\chi(\xi)E_{1/\rho,2}(-\xii^2)\bigr)
    & = 2\rho\,\tilde K_{1/\rho,-2\rho} + \frac1{\Gamma(2-1/\rho)}\,\mathfrak{F}^{-1}(\chi(\xi)\,\xii^{-2}) \\
    & \qquad + \pi^{-1}\,\sin(\pi(2-2/\rho))\,\int_0^\infty \frac{s^{2/\rho-2}}{s^{2/\rho}+2\cos(\pi/\rho)\,s^{1/\rho}+1}\,\tilde H_{1/\rho,-2\rho}(s,\cdot)\,ds \\
    & \qquad + \pi^{-1}\,\sin(\pi(2-1/\rho))\,\int_0^\infty \frac{s^{3/\rho-2}}{s^{2/\rho}+2\cos(\pi/\rho)\,s^{1/\rho}+1}\,\tilde H_{1/\rho,-2\rho}(s,\cdot)\,ds,
\end{align*}
due to~$m=1$ (incidentally, we notice that in the special case~$\rho=3/2$, the first integral in the representation given by Theorem~\ref{thm:ML} disappears, due to $\sin(2-2/\rho)=-\sin\pi=0$).

In view of Remark~\ref{rem:entire}, we have to prove the statement for the localized Riesz potential~$\mathfrak{F}^{-1}(\chi(\xi)\,\xii^{-2})$ and for the integral terms.

\bigskip

We notice that~$\mathfrak{F}^{-1}(\chi(\xi)\,\xii^{-2})$ behaves as the Bessel potential~$(1-\Delta)^{-1}$, due to the presence of the cut-off function, in particular, it belongs to~$L^p$, for any~$p$ such that~\eqref{eq:boundLp2} holds. To manage the two integrals, we split them into two parts:
\begin{align*}
I_j^-(x)
    & =\int_0^1 \frac{s^{(1+j)/\rho-2}}{s^{2/\rho}+2\cos(\pi/\rho)\,s^{1/\rho}+1}\,\tilde H_{1/\rho,-2\rho}(s,x)\,ds \\
I_j^+(x)
    & =\int_1^\infty \frac{s^{(1+j)/\rho-2}}{s^{2/\rho}+2\cos(\pi/\rho)\,s^{1/\rho}+1}\,\tilde H_{1/\rho,-2\rho}(s,x)\,ds
\end{align*}
for~$j=1,2$. Then we use Lemmas~\ref{lem:Htilde} and~\ref{lem:Htilde2} in~$I_j^-$, and Lemma~\ref{lem:Htilde3} in~$I_j^+$.

Let~$p$ be as in~\eqref{eq:boundLp2} and assume that either~\eqref{eq:Lpd} or~\eqref{eq:Lpdback} holds, with~$d=-2\rho$. Indeed, if~$p$ verifies the equality in~\eqref{eq:Lpdback} and~\eqref{eq:Lpd}, the proof will follow by interpolation.

If~\eqref{eq:Lpdback} holds with~$d=-2\rho$, then we may apply Lemma~\ref{lem:Htilde}, obtaining:
\[ \|I_j^-\|_{L^p} \lesssim \int_0^1 s^{(1+j)/\rho-2}\,ds, \]
for~$j=1,2$. Clearly, both the integrals converge. Let~$p$ be as in~\eqref{eq:boundLp2}, and assume that~\eqref{eq:Lpd} holds with~$d=-2\rho$. Then we may apply Lemma~\ref{lem:Htilde2}, obtaining:
\[ \|I_j^-\|_{L^p} \lesssim \int_0^1 \frac{s^{(1+j)/\rho-1-\frac{n}{2\rho}\left(1-\frac1p\right)}\,}{s^{2/\rho}+2\cos(\pi/\rho)\,s^{1/\rho}+1}\,ds, \]
for~$j=1,2$. These integral converge, for any~$p\in(1,\infty]$ such that
\[ \frac{n}{2}\left(1-\frac1p\right)<1+j, \qquad j=1,2. \]
In particular, they converge, since~$p$ verifies~\eqref{eq:boundLp2}.

On the other hand, by Lemma~\ref{lem:Htilde3} with~$d=-2\rho$, it follows that the integrals
\[ \int_1^\infty \frac{s^{(1+j)/\rho}\,e^{-cs}}{s^{2/\rho}+2\cos(\pi/\rho)\,s^{1/\rho}+1}\,ds, \qquad j=1,2, \]
converge, due to the exponential term~$e^{-cs}$.
\end{proof}

%
%When~\eqref{eq:uniformEalphaplus} does not hold, but~\eqref{eq:uniformE1plus} does, we may rely on the following:
%%
%\[ \int_0^t (t-s)^\alpha E_{1+\alpha,1+\alpha}(-(t-s)^{1+\alpha}\xii^2)\,f(s)\,ds = \int_0^t E_{1+\alpha,1}(-(t-s)^{1+\alpha}\xii^2)\,J_{0+}^\alpha f(s)\,ds, \]
%%
%which is a consequence of fractional integration by parts~\cite{Samko}, being
%%
%\[ \int_0^t f(s)\, J_{0+}^{\alpha} g(s)\,ds = \int_0^t g(s)\,J_{t-}^{\alpha} f(s)\, ds, \]
%%
%where
%%
%\[ J_{t-}^{\alpha} h (s) = \frac1{\Gamma(1-\alpha)}\,\int_s^t (t-\tau)^{-(1-\alpha)}\,h(\tau)\,d\tau. \]
%%
%Indeed (see 2.1.53 in~\cite{KST}):
%%
%\[ J_{0+}^\alpha \bigl( t^{\beta-1}\,E_{\gamma,\beta}(\lambda(t-s)^\gamma)\bigr) = t^{\alpha+\beta-1}E_{\gamma,\alpha+\beta}(\lambda(t-s)^\gamma), \]
%%
%for any~$\lambda\in\R$, $\alpha,\beta>0$, $\gamma\geq0$.

\subsection{Proof of Theorem~\ref{thm:linear}}

We are now ready to prove Theorem~\ref{thm:linear}. As it is customary, we rely on the following estimate.
\begin{lemma}\label{lem:integral}
Let~$a<1$ and~$b\in\R$. Then:
\begin{equation}
\label{eq:intgood}
\int_0^t (t-s)^{-a}\,(1+s)^{-b}\,ds \lesssim \begin{cases}
(1+t)^{-a}, &\qquad \text{if~$a<1<b$,}\\
(1+t)^{-1}\,\log(1+t), &\qquad \text{if~$a<1=b$,}\\
(1+t)^{1-a-b}, &\qquad \text{if~$a,b<1$.}
\end{cases}
\end{equation}
\end{lemma}
For the ease of reading, we provide a proof of Lemma~\ref{lem:integral}, even if it is standard and well-known.
\begin{proof}
For~$t\in[0,1]$, it is sufficient to estimate
\[ \int_0^t (t-s)^{-a}\,(1+s)^{-b}\,ds \leq \int_0^t (t-s)^{-a}\,ds = \frac{t^{1-a}}{1-a} \leq \frac1{1-a},\]
whereas, for~$t\geq1$ we split the integration interval into~$[0,t/2]$ and~$[t/2,t]$, deriving
\[ \int_0^t (t-s)^{-a}\,(1+s)^{-b}\,ds \approx t^{-a} \int_0^{t/2}(1+s)^{-b}\,ds + t^{-b}\int_{t/2}^t (t-s)^{-a}\,ds, \]
thanks to~$t-s\in[t/2,t]$, for any~$s\in[0,t/2]$ and~$s\in[t/2,t]$ for any~$s\in[t/2,t]$. The proof follows from:
\[ \int_{t/2}^t (t-s)^{-a}\,ds = \frac{(t/2)^{1-a}}{1-a}, \]
and
\[ \int_0^{t/2}(1+s)^{-b}\,ds \lesssim \begin{cases}
1 & \text{if~$b>1$,}\\
\log (1+t) & \text{if~$b=1$,}\\
(1+t)^{1-b} & \text{if~$b<1$.}
\end{cases} \]
This concludes the proof of~\eqref{eq:intgood}.
\end{proof}
\begin{proof}[Proof of Theorem~\ref{thm:linear}]
By virtue of~\eqref{eq:scalingt} and Young inequality, we have that:
\begin{equation}
\label{eq:Gest}
\|G_{1/\rho,\beta}(t,\cdot)\ast h\|_{L^{q}} \lesssim t^{-\frac{n}{2\rho}\left(\frac1r-\frac1q\right)}\,\|h\|_{L^r},
\end{equation}
with~$\beta=1,1/\rho,2$, for any~$1\leq r\leq q\leq \infty$, such that
\begin{equation}\label{eq:boundLrq}
\frac{n}{2}\left(\frac1r-\frac1q\right)< \begin{cases}
1 & \text{if~$\beta=1, 2$,}\\
2 & \text{if~$\beta=1/\rho$.}
\end{cases}
\end{equation}
The proof of the estimate for the homogeneous part of the solution~$u^\hom$ immediately follows. For the inhomogeneous part of the solution related to the term~$f(t,x)$, it is sufficient to apply Lemma~\ref{lem:integral}, with
\[ a=\frac{n(1+\alpha)}2\left(\frac1{r_2}-\frac1q\right)-\alpha, \]
and~$b=\eta$.
\end{proof}

%%%%%%%%%%%%%%%%%%%%%%%%%%%%%%%%%%%%%%%%%%%%

\section{Proof of the global existence results}\label{sec:global}

We are now ready to prove Theorems~\ref{thm:10} and~\ref{thm:00}.

By~\eqref{eq:solutiondef}, a function $u\in X$, where~$X$ is a suitable space, is a solution to~\eqref{eq:CP} if, and only if, it satisfies the equality
\begin{equation}\label{eq:fixedpoint}
u(t,x) =  u^\lin(t,x) + Nu(t,x), \qquad \text{in~$X$,}
\end{equation}
where we set $u^\lin=u^\hom$, with $u^\hom$ as in~\eqref{eq:uhom} and~$f(s,x)=|u(s,x)|^p$, so that
\[ Nu(t,x)= \int_0^t (t-s)^\alpha\,G_{1+\alpha,1+\alpha}(t-s,x)\ast_{(x)}|u(s,x)|^p\,ds. \]
We will use the notation~$u^\lin$ instead of~$u^\hom$ in this Section, since this is now the linear part of the solution to the semilinear problem~\eqref{eq:CP}, whereas it was before the homogeneous part of the solution to the linear problem~\eqref{eq:CPlinear}.

The proof of our global existence results is based on the following scheme. We define~$X$ as the subspace of~$\mathcal C([0,\infty),L^1\cap L^p)$ for which a suitable norm~$\|\cdot\|_X$ is finite. This norm is related to the desired decay rates for the solution to~\eqref{eq:CP}. In particular, we show that~$u^\lin\in X$, and that
\begin{equation}\label{eq:estulin}
\|u^\lin\|_X \leq C\,\|u_0\|_{L^1\cap L^p},
\end{equation}
then we prove the estimates
\begin{align}
\label{eq:well}
\|Nu\|_{X}
    & \leq C\|u\|_{X}^p\,, \\
\label{eq:contraction}
\|Nu-Nv\|_{X}
    & \leq C\|u-v\|_{X} \bigl(\|u\|_{X}^{p-1}+\|v\|_{X}^{p-1}\bigr)\,.
\end{align}
By standard arguments, since $u^\lin\in X$ and~$p>1$, from~\eqref{eq:well} it follows that~$u^\lin+Nu$ maps balls of~$X$ into balls of~$X$, for small data in~$L^1\cap L^p$, and that estimates \eqref{eq:well}-\eqref{eq:contraction} lead to the existence of a unique solution~$u$ to~\eqref{eq:fixedpoint}. We simultaneously gain a local and a global existence result.

Our starting point is the use of the linear estimates in Theorem~\ref{thm:linear}. For both Theorems~\ref{thm:10} and~\ref{thm:00}, we prove~\eqref{eq:well}, but we omit the proof of~\eqref{eq:contraction}, since it is analogous to the proof of~\eqref{eq:well}. For the ease of reading, we first prove the simpler Theorem~\ref{thm:00}.
\begin{proof}[Proof of Theorem~\ref{thm:00}]
We define
\[ X= \{ u\in \mathcal C([0,\infty), L^1\cap L^p): \ \|u\|_X<\infty\}, \]
with norm:
\[ \|u\|_X = \sup_{t\geq0} (1+t)^{-\alpha}\,\bigl\{\|u(t,\cdot)\|_{L^1} + (1+t)^{\frac{n}2(1+\alpha)\left(1-\frac1p\right)} \|u(t,\cdot)\|_{L^p} \bigr\}. \]
For any~$q\in[1,p]$, we define
\[ \beta_q = \frac{n}2(1+\alpha)\left(1-\frac1{q}\right), \]
as in~\eqref{eq:betaqgood}. By interpolation, a function~$u\in X$ verifies
\[ \|u(t,\cdot)\|_{L^q}\leq (1+t)^{\alpha-\beta_q}\,\|u\|_X, \qquad \forall q\in[1,p]. \]
Thanks to Theorem~\ref{thm:linear}, the linear part~$u^\lin$ of the solution is in~$X$, and~\eqref{eq:estulin} holds. Indeed, taking~$r_0=q=1$, it holds
\[ \|u^\lin(t,\cdot)\|_{L^1} \lesssim \|u_0\|_{L^1}. \]
On the other hand, if~$q=p$ then we take~$r_0=p$ for~$t\in[0,1]$ and~$r_0=1$ for~$t\geq1$, so that
\[ \|u^\lin(t,\cdot)\|_{L^p} \lesssim \begin{cases}
\|u_0\|_{L^p} & t\in[0,1],\\
t^{-\frac{n}2(1+\alpha)\left(1-\frac1p\right)}\|u_0\|_{L^1} & t\geq1,
\end{cases} \]
and this leads to
\[ \|u^\lin(t,\cdot)\|_{L^p} \lesssim (1+t)^{-\frac{n}2(1+\alpha)\left(1-\frac1p\right)}\bigl(\|u_0\|_{L^1}+\|u_0\|_{L^p}\bigr), \qquad t\geq0. \]
Now we consider the nonlinear part of the solution. For any~$u\in X$, we get
\[ \||u(t,\cdot)|^p\|_{L^1} \lesssim \|u(t,\cdot)\|_{L^p}^p \lesssim (1+t)^{-\frac{n}2(1+\alpha)(p-1)+p\alpha} \|u\|_X^p. \]
In particular,
\[ \eta=\frac{n}2(1+\alpha)(p-1)-p\alpha>1, \]
if, and only if, $p>\tilde p(n,\alpha)$. We now apply Theorem~\ref{thm:linear} to the nonlinear part of the solution, i.e., we set~$f(t,x)=|u(t,x)|^p$ and~$K=c\|u\|_X^p$, for some~$c>0$. Let~$q=1,p$. By taking~$r_2=1$, we derive
\begin{align*}
\|Nu(t,\cdot)\|_{L^1}
    & \lesssim (1+t)^{\alpha}\|u\|_X^p,\\
\|Nu(t,\cdot)\|_{L^p}
    & \lesssim (1+t)^{\alpha-\frac{n}2(1+\alpha)\left(1-\frac1p\right)}\|u\|_X^p,
\end{align*}
so that~$Nu\in X$ and~$\|Nu\|_X \lesssim \|u\|_X^p$, i.e. we obtain~\eqref{eq:well}. This concludes the proof.
\end{proof}
We now prove Theorem~\ref{thm:10}.
\begin{proof}[Proof of Theorem~\ref{thm:10}]
We now define
\[ X= \{ u\in \mathcal C([0,\infty), L^1\cap L^p): \ \|u\|_X<\infty\}, \]
with norm:
\[ \|u\|_X = \sup_{t\geq0} (1+t)^{-1}\,\bigl\{\|u(t,\cdot)\|_{L^1} + (1+t)^{\frac{n}2(1+\alpha)\left(1-\frac1p\right)} \|u(t,\cdot)\|_{L^p} \bigr\}. \]
For any~$q\in[1,p]$, we define
\[ \beta_q = \frac{n}2(1+\alpha)\left(1-\frac1{q}\right), \]
as in~\eqref{eq:betaqgood}. By interpolation, a function~$u\in X$ verifies
\[ \|u(t,\cdot)\|_{L^q}\leq (1+t)^{1-\beta_q}\,\|u\|_X, \qquad \forall q\in[1,p]. \]
Thanks to Theorem~\ref{thm:linear}, the linear part~$u^\lin$ of the solution is in~$X$, and~\eqref{eq:estulin} holds. Indeed, taking~$r_0=q=1$, it holds
\[ \|u^\lin(t,\cdot)\|_{L^1} \lesssim \|u_0\|_{L^1}+t\,\|u_1\|_{L^1}, \]
in particular,
\[ \|u^\lin(t,\cdot)\|_{L^1} \lesssim (1+t)\bigl(\|u_0\|_{L^1}+\|u_1\|_{L^1}\bigr). \]
On the other hand, if~$q=p$ then we take~$r_0=p$ for~$t\in[0,1]$ and~$r_0=1$ for~$t\geq1$, so that
\[ \|u^\lin(t,\cdot)\|_{L^p} \lesssim \begin{cases}
\|u_0\|_{L^p}+t\|u_1\|_{L^p} & t\in[0,1],\\
t^{-\frac{n}2(1+\alpha)\left(1-\frac1p\right)}\bigl(\|u_0\|_{L^1}+t\|u_1\|_{L^1}\bigr) & t\geq1,
\end{cases} \]
and this leads to
\[ \|u^\lin(t,\cdot)\|_{L^p} \lesssim (1+t)^{1-\frac{n}2(1+\alpha)\left(1-\frac1p\right)}\bigl(\|u_0\|_{L^1}+\|u_1\|_{L^1}+\|u_0\|_{L^p}+\|u_1\|_{L^p}\bigr), \qquad t\geq0. \]
Now we consider the nonlinear part of the solution. For any~$u\in X$, we get
\[ \||u(t,\cdot)|^p\|_{L^1} \lesssim \|u(t,\cdot)\|_{L^p}^p \lesssim (1+t)^{-\frac{n}2(1+\alpha)(p-1)+p} \|u\|_X^p. \]
In particular,
\[ \eta=\frac{n}2(1+\alpha)(p-1)-p\geq\alpha, \]
if, and only if, $p\geq \bar p(n,\alpha)$. We now apply Theorem~\ref{thm:linear} to the nonlinear part of the solution, i.e., we set~$f(t,x)=|u(t,x)|^p$ and~$K=c\|u\|_X^p$, for some~$c>0$. By taking~$r_2=1$, we derive
\[
\|Nu(t,\cdot)\|_{L^1} \lesssim \begin{cases}
    (1+t)^{\alpha}\|u\|_X^p & \text{if~$\eta>1$}\\
    (1+t)^\alpha\log (e+t)\|u\|_X^p & \text{if~$\eta=1$}\\
    (1+t)^{\alpha+1-\eta}\|u\|_X^p & \text{if~$\eta<1$}
    \end{cases}
\]
In particular, the assumption~$\eta\geq\alpha$ guarantees that
\[ \|Nu(t,\cdot)\|_{L^1} \lesssim (1+t)\|u\|_X^p. \]
We proceed similarly to derive
\[ \|Nu(t,\cdot)\|_{L^p} \lesssim (1+t)^{1-\frac{n}2(1+\alpha)\left(1-\frac1p\right)}\|u\|_X^p, \]
so that~$Nu\in X$ and~$\|Nu\|_X \lesssim \|u\|_X^p$, i.e. we obtain~\eqref{eq:well}. This concludes the proof.
\end{proof}

%%%%%%%%%%%%%%%%%%%%%%%%%%%%%%%%%%%%%%%%%%%%%%%%%%

\section{Concluding remarks}\label{sec:concluding}

In order to keep the structure of the paper as simpler as possible, we postponed in this section some additional results which are not essential to the main purpose of the paper, but that can be of some interest for the reader.

\subsection{Extending the range for~$p$ beyond~$1+2/(n-2)$}\label{sec:Linf}

If~$n\geq3$ and~$p\geq 1+2/(n-2)$, then one may easily extend Theorems~\ref{thm:10} and~\ref{thm:00}. To avoid formal difficulties, it is convenient to state a result where data are assumed to be small in~$L^1\cap L^\infty$.
\begin{theorem}\label{thm:10inf}
Let~$n\geq2$ and $p\geq\bar p$, in~\eqref{eq:CP}, with~$\bar p$ as in~\eqref{eq:pcrit}. Then there exists~$\eps>0$ such that for any~$u_0,u_1\in L^1\cap L^\infty$, satisfying
\begin{align}
\label{eq:u0epsinf}
\|u_0\|_{L^1\cap L^\infty}
    & \doteq \|u_0\|_{L^1}+\|u_0\|_{L^\infty}\leq \eps,\\
\label{eq:u1epsinf}
\|u_1\|_{L^1\cap L^\infty}
    & \doteq \|u_1\|_{L^1}+\|u_1\|_{L^\infty}\leq \eps,
\end{align}
there exists a unique global solution
\begin{equation}\label{eq:solution0inf}
u\in\mathcal C([0,\infty),L^1\cap L^\infty)
\end{equation}
to~\eqref{eq:CP}. Moreover, for any~$\delta>0$, the solution satisfies
\begin{equation}
\label{eq:decayu1inf}
\|u(t,\cdot)\|_{L^q} \leq C\,(1+t)^{1-\beta_q}\,\bigl(\|u_0\|_{L^1\cap L^\infty}+\|u_1\|_{L^1\cap L^\infty}\bigr), \qquad \forall \, q\in[1,\infty], \ \forall t\geq0,
\end{equation}
where~$\beta_q$ is as in~\eqref{eq:betaq}, and~$C>0$ does not depend on the data.
\end{theorem}
\begin{theorem}\label{thm:00inf}
Let~$n\geq1$, $u_1=0$ and~$p>\tilde p$ in~\eqref{eq:CP}, with~$\tilde p$ as in~\eqref{eq:pcrit0}. Then there exists~$\eps>0$ such that for any~$u_0\in L^1\cap L^\infty$, satisfying~\eqref{eq:u0epsinf}, there exists a unique global solution as in~\eqref{eq:solution0inf}, to~\eqref{eq:CP}. Moreover, for any~$\delta>0$, the solution satisfies the following estimates:
\begin{equation}
\label{eq:decayu0inf}
\|u(t,\cdot)\|_{L^q} \leq C\,(1+t)^{\alpha-\beta_q}\,\|u_0\|_{L^1\cap L^\infty}, \qquad \forall \, q\in[1,\infty], \ \forall t\geq0,
\end{equation}
where~$\beta_q$ is as in~\eqref{eq:betaq}, and~$C>0$ does not depend on the data.
\end{theorem}
In order to prove Theorems~\ref{thm:10inf} and~\ref{thm:00inf}, it is sufficient to follow the proof of Theorems~\ref{thm:10} and~\ref{thm:00}, with minor modifications. One may fix
\[ X= \{ u\in \mathcal C([0,\infty), L^1\cap L^\infty): \ \|u\|_X<\infty\}, \]
with norm:
\[ \|u\|_X = \sup_{t\geq0} \bigl\{(1+t)^{-1}\,\|u(t,\cdot)\|_{L^1} + (1+t)^{\alpha-\delta} (\|u(t,\cdot)\|_{L^{q_0}}+\|u(t,\cdot)\|_{L^\infty}) \bigr\}, \]
in Theorem~\ref{thm:10inf}, and
\[ \|u\|_X = \sup_{t\geq0} \bigl\{(1+t)^{-\alpha}\,\|u(t,\cdot)\|_{L^1} + (1+t)^{1-\delta} (\|u(t,\cdot)\|_{L^{q_0}}+\|u(t,\cdot)\|_{L^\infty}) \bigr\}, \]
in Theorem~\ref{thm:00inf} (with~$n\geq2$), where~$q_0=q_0(\delta)\in(1,\infty)$ verifies
\begin{equation}\label{eq:q0}
\frac{n}2\left(1-\frac1{q_0}\right) = 1-\frac\delta{1+\alpha},
\end{equation}
for a sufficiently small~$\delta>0$. Then one obtains the desired result applying Theorem~\ref{thm:linear} to~$u^\lin$ and to the nonlinear part of the solution, by choosing, time by time, suitable~$r_0,r_1,r_2$, so that~\eqref{eq:r} is verified.

%%%%%%%%%%%%%%%%%%%%%%%%%%%%%%%%%%%%%%%%%%%%%%%%%%

\subsection{Estimates for the spatial derivatives of the solution}\label{sec:spatial}

With minor modifications in Section~\ref{sec:linear}, it is possible to prove that
\[ (-\Delta)^{\frac{\kappa}{2}}\,G_{1/\rho,\beta}(1,\cdot) \in L^p, \]
where~$\kappa>0$, provided that
\[ \frac{n}2\left(1-\frac1p\right) + \frac\kappa2 < \begin{cases}
1 & \text{if~$\beta=1,2$,}\\
2 & \text{if~$\beta=1/\rho$.}
\end{cases} \]
Indeed, it is sufficient to fix~$d=\kappa$ for $\beta=1$  in \eqref{eq:Kdef} {and} \eqref{eq:Hdef}, and $d=\kappa-2(1-\rho),\kappa-2\rho$, respectively, for~$\beta=1/\rho,2$ in~\eqref{eq:Ktildedef} and~\eqref{eq:Htildedef}, and consequently modify Propositions~\ref{prop:E1}, \ref{prop:Ealpha} and~\ref{prop:E2}. Indeed, one may easily prove that
\[ \|(-\Delta)^{{\frac{\kappa}{2}}}\,\partial_x^\gamma\,G_{1/\rho,\beta}(t,\cdot)\|_{L^{p}}\lesssim t^{-\frac{n}{2\rho}\left(1-\frac{1}{p} \right)-\frac{\kappa+|\gamma|}{2\rho}}, \]
where~$\kappa>0$ and~$\gamma\in\N^n$, provided that
\[ \frac{n}2\left(1-\frac1p\right) + \frac{\kappa + |\gamma|}2 < \begin{cases}
1 & \text{if~$\beta=1,2$,}\\
2 & \text{if~$\beta=1/\rho$.}
\end{cases} \]
These results allow to obtain estimates for the spatial derivatives of the solution to~\eqref{eq:CPlinear} and, therefore, to~\eqref{eq:CP}. Also, nonlinearities like~$|\nabla u|^p$, or~$\Delta (|u|^p)$ may be considered. As a mere example, we provide a result for~$\nabla u$. % and a result for~$\Delta u$.
\begin{theorem}\label{thm:nabla}
Let~$n\geq1$ and~$q\in[1,\infty]$. Assume that~$u_0,\nabla u_0\in L^{r_0}$, $u_1\in L^{r_1}$, and that~$f(t,\cdot)\in L^{r_2}$, with~$r_j\in[1,q]$, satisfying
\begin{equation}\label{eq:rnabla}
n\left(\frac1{r_j}-\frac1q\right) < 1,
\end{equation}
for~$j=0,1,2$. Assume that~\eqref{eq:fgood} holds for some~$K>0$ and~$\eta\in\R$. Then the solution to~\eqref{eq:CPlinear} verifies the following estimate:
\begin{align*}
\|\nabla u(t,\cdot)\|_{L^q}
    & \leq C\, t^{-\frac{n(1+\alpha)}2\left(\frac1{r_0}-\frac1q\right)}\, (1+t)^{-\frac{1+\alpha}2}\,\bigl(\|u_0\|_{L^{r_0}}+\|\nabla u_0\|_{L^{r_0}}\bigr)\\
    & \qquad + C\,t^{\frac{1-\alpha}2-\frac{n(1+\alpha)}2\left(\frac1{r_1}-\frac1q\right)}\, \|u_1\|_{L^{r_1}}\\
    & \qquad + \begin{cases}
    CK\,(1+t)^{-\frac{1-\alpha}2-\frac{n(1+\alpha)}2\left(\frac1{r_2}-\frac1q\right)} & \text{if~$\eta>1$,} \\
    CK\,(1+t)^{-\frac{1-\alpha}2-\frac{n(1+\alpha)}2\left(\frac1{r_2}-\frac1q\right)}\,\log(1+t) & \text{if~$\eta=1$,} \\
    CK\,(1+t)^{\frac{1+\alpha}2-\eta-\frac{n(1+\alpha)}2\left(\frac1{r_2}-\frac1q\right)} & \text{if~$\eta<1$,}
    \end{cases}
\end{align*}
for any~$t>0$, where~$C$ does not depend on the data.
\end{theorem}
\begin{remark}
The assumption~$\nabla u_0\in L^{r_0}$ is taken to give a non-singular estimate at~$t=0$, when~$q=r_0$, namely, to guarantee the well-posedness of the homogeneous problem in~$L^q$. Indeed, for~$f\equiv0$, one has
\[ \|\nabla u(t,\cdot)\|_{L^q} \leq C\, (1+t)^{-\frac{1+\alpha}2}\,\bigl(\|u_0\|_{L^q}+\|\nabla u_0\|_{L^q}\bigr) + C\,t^{\frac{1-\alpha}2}\, \|u_1\|_{L^q}.\]
\end{remark}
\begin{proof}
The proof is analogous to the proof of Theorem~\ref{thm:linear}, but~\eqref{eq:Gest} is replaced by
\[ \|\nabla G_{1/\rho,1}(t,\cdot)\ast f\|_{L^{q}} \lesssim t^{-\frac{n}{2\rho}\left(\frac1r-\frac1q\right)-\frac1{2\rho}}\,\|f\|_{L^r}. \]
However, in the estimate with respect to~$u_0$, for~$t\leq1$, the gradient is applied to~$u_0$, i.e., \eqref{eq:Gest} is modified into
\[ \|\nabla G_{1/\rho,1}(t,\cdot)\ast f\|_{L^{q}} \lesssim t^{-\frac{n}{2\rho}\left(\frac1{r_0}-\frac1q\right)}\,\|\nabla f\|_{L^{r_0}}. \]
\end{proof}

\subsection{Estimates for the time derivatives of the solution}\label{sec:time}

By using the following formula for derivatives of the Mittag-Leffler functions (see (1.10.7) in~\cite{KST}):
\[ \partial_z^n \bigl( z^{\beta-1}E_{\alpha+1,\beta}(\lambda z^{\alpha+1})\bigr) = z^{\beta-n-1}\,E_{\alpha+1,\beta-n}(\lambda z^{\alpha+1}). \]
one may derive estimates for the time-derivatives of the solution to~\eqref{eq:CPlinear}. In particular, since~$u$ in~\eqref{eq:solutiondef} solves~\eqref{eq:CPlinear}, then
\begin{equation}\label{eq:solutionderdef}
u_t(t,\cdot)=u_t^\hom(t,\cdot) + \int_0^t (t-s)^{\alpha-1}\,G_{1+\alpha,\alpha}(t-s,\cdot) \ast_{(x)} f(s,\cdot)\,ds,
\end{equation}
where
\begin{equation}\label{eq:uthom}
u_t^\hom(t,\cdot)= t^{-1}\,G_{1+\alpha,0}(t,\cdot) \ast_{(x)} u_0 + G_{1+\alpha,1}(t,\cdot) \ast_{(x)} u_1.
\end{equation}
The mapping properties of~$G_{1+\alpha,1}(1,\cdot)$ are studied in Proposition~\ref{prop:E1}, whereas the mapping properties of~$G_{1+\alpha,0}(1,\cdot)$ and~$G_{1+\alpha,\alpha}(1,\cdot)$ may be easily studied using once again the representation in Theorem~\ref{thm:ML}. In particular,
\[ \|u_t^\hom(t,\cdot)\|_{L^q} \lesssim \|(-\Delta)^{\frac1{1+\alpha}}u_0\|_{L^q} + \|u_1\|_{L^q}. \]
Once linear estimates are obtained, they may be included in the statements of the nonlinear results, and nonlinearities like~$|u_t|^p$ may also be studied.

\subsection{Stronger smallness assumption on~$u_1$}

The global existence exponent in Theorem~\ref{thm:10} may be improved, if stronger smallness assumption are taken for the second data~$u_1$. In particular, if~$(-\Delta)^{-\frac\kappa2}u_1 \in L^1$ for some~$\kappa\in(0,2)$, then
\[ \|G_{1+\alpha,2}(t,\cdot)\ast_{(x)}u_1\|_{L^q} \lesssim t^{1-\frac{n}2(1+\alpha)\left(1-\frac1q\right)-\frac\kappa2\,(1+\alpha)}\|(-\Delta)^{-\frac\kappa2}u_1\|_{L^1}. \]
for any~$q\geq1$, such that
\[ \frac{n}2\left(1-\frac1q\right)+\frac\kappa2 <1, \]
thanks to the mapping properties of~$(-\Delta)^{\frac\kappa2} G_{1+\alpha,2}$ (see Section~\ref{sec:spatial}). We recall that
\[ \|(-\Delta)^{-\frac\kappa2}u_1\|_{L^1} \leq \|(-\Delta)^{-\frac\kappa2}u_1\|_{H^1} \lesssim \|u_1\|_{H^{r_1}}, \]
where~$r_1\in(0,1)$ is defined as
\[ n\left(\frac1{r_1}-1\right)=\kappa, \]
and~$H^{r}$ is the real Hardy space of exponent~$r\in(0,1)$.

\bigskip

In particular, taking~$\kappa=2(1-\alpha)/(1+\alpha)$, the decay rate for the solution to~\eqref{eq:CPlinear}, with respect to the second data, becomes the same one obtained when~$u_1\equiv0$. In turn, this leads to improve the critical exponent for~\eqref{eq:CP} to~$\tilde p(n,\alpha)$, by replacing assumption~\eqref{eq:u1eps} in Theorem~\ref{thm:10} with
\[ \|u_1\|_{H^{r_1}}+\|u_1\|_{L^\infty}\leq \eps. \]
This result does not contradict the nonexistence result in~\cite{DA+} since any function in a real Hardy space~$H^r$, with~$r\in(0,1]$, verifies the moment condition, that is, its integral is zero, so that no sign assumption on~$u_1$ is possible. We address the reader interested in decay estimates in real Hardy spaces for damped evolution equations to~\cite{DAEP16}.

\subsection{Estimates for higher order equations}

If we consider the equation in~\eqref{eq:CPlinear} with
\[ \partial_t^{1+\alpha} u +(-\Delta)^m u = f(t,x), \]
where~$m\in\N\setminus\{0,1\}$ (or even~$m\in\R$, $m>0$), then it is sufficient to modify the definition of~~$G_{1/\rho,\beta}$ in~\eqref{eq:Gdef}, setting
\[ G_{1/\rho,\beta}(t,x)= \mathfrak{F}^{-1} \bigl(E_{1/\rho,\beta}(-t^{1/\rho}\xii^{2m})\bigr). \]
With minor modifications in Section~\ref{sec:linear}, it is possible to prove that now
\[ G_{1/\rho,\beta}(1,\cdot) \in L^p, \]
provided that
\[ \frac{n}{2m}\left(1-\frac1p\right) < \begin{cases}
1 & \text{if~$\beta=1,2$,}\\
2 & \text{if~$\beta=1/\rho$.}
\end{cases} \]
The statements of Theorems~\ref{thm:linear}, \ref{thm:00} and~\ref{thm:10} are consequently modified. In particular, the global existence of small data solutions holds for~$p>\tilde p(n/m,\alpha)$ if~$u_1\equiv0$, and~$p>\bar p(n/m,\alpha)$ otherwise.

\subsection{The Cauchy problem with Riemann-Liouville fractional derivative}

By using the mapping properties for~$G_{1/\rho,1/\rho}$ derived in Proposition~\ref{prop:Ealpha}, it is possible to study the easier problem of the Cauchy problem for the fractional diffusive equation
\begin{equation}\label{eq:CPlinearRL}
\begin{cases}
D^{1+\alpha} u -\Delta u = f(t,x),\qquad t>0,\ x\in\R^n,\\
J^{1-\alpha} u (0,x)=v_{1-\alpha}(x),\\
D^\alpha u(0,x)=u_\alpha(x),
\end{cases}
\end{equation}
where~$D^\beta$ is the Riemann-Liouville fractional derivative defined in~\eqref{eq:RLder}. Indeed, the solution to~\eqref{eq:CPlinearRL}
%
%
%\begin{equation}\label{eq:CPFRL}
%\begin{cases}
%D_t^{1+\alpha} \hat u + \xii^{2} \hat u = \hat f(t,\xi),\qquad t\geq0,\ x\in\R^n,\\
%D_t^\alpha \hat u(0,\xi)=\hat u_\alpha(\xi),
%\end{cases}
%\end{equation}
%
is now given by (see Theorem 4.1 and Example 4.2 in~\cite{KST}):
\begin{align*}%\label{eq:solutiondefRL}
u(t,\cdot)
    & =t^{-(1-\alpha)}\,G_{1+\alpha,\alpha}(t,\cdot) \ast_{(x)} v_{1-\alpha} + t^\alpha\,G_{1+\alpha,1+\alpha}(t,\cdot) \ast_{(x)} u_\alpha \\
    & \qquad + \int_0^t (t-s)^{\alpha}\,G_{1+\alpha,1+\alpha}(t-s,\cdot) \ast_{(x)} f(s,\cdot)\,ds,
\end{align*}
where~$G_{1+\alpha,\beta}$ is as in~\eqref{eq:Gdef}. A result similar to Theorem~\ref{thm:linear} may then be easily obtained when~$v_{1-\alpha}\equiv0$, and applied to study the problem with power nonlinearity~$|u|^p$. For this problem, the critical exponent is the one given by scaling arguments (see Section~\ref{sec:scaling}), i.e. $\tilde p(n,\alpha)$. Indeed,
\[ D^\alpha \bigl( u(\lambda^{\frac2{1+\alpha}}t,\lambda x)\bigr)\bigl|_{t=0}\bigr. = \lambda^{\frac{2\alpha}{1+\alpha}}\,u_\alpha(\lambda x), \]
so that the solution to~$q_\sca=1$ is given by~$p=\tilde p(n,\alpha)$. The fact that the critical exponent is the expected one from scaling arguments, is consistent with the fact that the solution to the nonlinear problem suffers no loss of decay with respect to the solution to the linear one.

\subsection{An extension of Theorem~\ref{thm:linear}}

As~$t\to\infty$, Lemma~\ref{lem:integral} may be extended to cover the case of different pairs of coefficients~$a,b$, in the integration ranges~$[0,t/2]$ and~$[t/2,t]$, in the following way.
\begin{lemma}\label{lem:integral1}
Let~$a_1<1$ and~$a_0, b_0, b_1\in\R$, and assume that~$k(t,s)$ is a nonnegative function, such that
\[ k(t,s)\leq \min\{ (t-s)^{-a_0}\,(1+s)^{-b_0}, \ (t-s)^{-a_1}\,(1+s)^{-b_1}\}. \]
Then, for any~$t\geq0$, it holds:
\begin{equation}
\label{eq:intrefined}
\int_0^t k(t,s)\,ds\lesssim (1+t)^{1-a_1-b_1} + \begin{cases}
(1+t)^{-a_0}, &\qquad \text{if~$b_0>1$,}\\
(1+t)^{-a_0}\,\log(1+t), &\qquad \text{if~$b_0=1$,}\\
(1+t)^{1-a_0-b_0}, &\qquad \text{if~$b_0<1$.}
\end{cases}
\end{equation}
\end{lemma}
\begin{proof}
The proof is analogous to the proof of Lemma~\ref{lem:integral}. For~$t\leq1$, we estimate
\[ \int_0^t k(t,s)\,ds \lesssim \int_0^t (t-s)^{-a_1}\,ds \leq C, \]
whereas, for~$t\geq1$, we estimate:
\begin{gather*}
\int_0^{t/2} (t-s)^{-a_0}\,(1+s)^{-b_0}\,ds \approx t^{-a_0} \int_0^{t/2} (1+s)^{-b_0}\,ds\approx \begin{cases}
t^{-a_0}, &\qquad \text{if~$b_0>1$,}\\
t^{-a_0}\,\log(1+t), &\qquad \text{if~$b_0=1$,}\\
t^{1-a_0-b_0}, &\qquad \text{if~$b_0<1$,}
\end{cases} \\
\int_{t/2}^t (t-s)^{-a_1}\,(1+s)^{-b_1}\,ds \approx t^{-b_1}\,\int_{t/2}^t (t-s)^{-a_1}\,ds \approx t^{1-a_1-b_1}.
\end{gather*}
This concludes the proof.
\end{proof}
Theorem~\ref{thm:linear} may then be consequently extended.
\begin{theorem}\label{thm:linear1}
Let~$n\geq1$ and~$q\in[1,\infty]$. Assume that~$u_0\in L^{r_0}$, $u_1\in L^{r_1}$, and that~$f(t,\cdot)\in L^{r_3}\cap L^{r_2}$, with~$r_j\in[1,q]$, satisfying~\eqref{eq:r} for~$j=0,1,2$, and
\begin{equation}\label{eq:r3}
1\leq \frac{n}2\left(\frac1{r_3}-\frac1q\right) < 2,
\end{equation}
for~$j=0,1,2$. Assume that
\begin{equation}\label{eq:fextened}
\|f(t,\cdot)\|_{L^{r_j}}\leq K\,(1+t)^{-\eta_j}, \qquad \forall t\geq0,
\end{equation}
for~$j=2,3$, for some~$K>0$ and~$\eta_2,\eta_3>1$. Moreover, assume that
\[ \frac{n}2(1+\alpha)\left(\frac1{r_3}-\frac1{r_2}\right) \leq \eta_2-1. \]
Then the solution to~\eqref{eq:CPlinear} verifies the following estimates:
\[ \|u(t,\cdot)\|_{L^q}\leq C\, t^{-\frac{n(1+\alpha)}2\left(\frac1{r_0}-\frac1q\right)}\, \|u_0\|_{L^{r_0}}+ C\,t^{1-\frac{n(1+\alpha)}2\left(\frac1{r_1}-\frac1q\right)}\, \|u_1\|_{L^{r_1}}+ CK\,(1+t)^{\alpha-\frac{n(1+\alpha)}2\left(\frac1{r_3}-\frac1q\right)}, \]
for any~$t>0$, where~$C$ does not depend on the data.
\end{theorem}
\begin{proof}
It is sufficient to follow the proof of Theorem~\ref{thm:linear}, using Lemma~\ref{lem:integral1} with
\begin{align*}
a_0
    & = \frac{n(1+\alpha)}2\left(\frac1{r_3}-\frac1q\right)-\alpha, \\
a_1
    & = \frac{n(1+\alpha)}2\left(\frac1{r_2}-\frac1q\right)-\alpha,
\end{align*}
and~$b_0=\eta_3$, $b_1=\eta_2$. Indeed, $a_1+b_1-1 \geq a_0$.
\end{proof}
\begin{remark}
Thanks to Theorem~\ref{thm:linear1}, it is possible to improve the decay rate of the solution in Theorem~\ref{thm:00} in space dimension~$n\geq2$, for~$q\geq 1+2/(n-2)$, namely, for any small~$\delta>0$, one may prove that
\[ \|u(t,\cdot)\|_{L^q} \lesssim (1+t)^{-\min\left\{\frac{n}2(1+\alpha)\left(1-\frac1q\right)-\alpha,1+\alpha-\delta\right\}}\,\|u_0\|_{L^1\cap L^\infty}. \]
We avoid the details, for the sake of brevity.
\end{remark}

\section*{Acknowledments}

The first author is member of the Gruppo Nazionale per l'Analisi Matematica, la Probabilit\`a e le loro Applicazioni (GNAMPA) of the Istituto Nazionale di Alta Matematica (INdAM). The first and the third authors have been supported by AUCANI/USP - Novas Parcerias Internacionais Programa USP-Santander de Mobilidade Docente and by INdAM - GNAMPA Project 2016 \emph{Profili asintotici per equazioni di tipo dispersivo.} The second and third authors are partially supported by S\~ao Paulo Research Foundation (Fapesp) grants 2015/16038-2 and 2013/17636-5, respectively.

\end{document}